\documentclass{amsart}
\usepackage{setspace, hyperref}
\usepackage[normalem]{ulem}
\usepackage{amsmath, amssymb,amsthm, amsfonts, color}
\usepackage[mathscr]{eucal}
\newtheorem{thm}{Theorem}[section]
\newtheorem{lemma}[thm]{Lemma}
\newtheorem{prop}[thm]{Proposition}
\newtheorem{cor}[thm]{Corollary}

\newcommand{\R}{\mathbb R}
\theoremstyle{remark}
\newtheorem{remark}[thm]{Remark}

\numberwithin{equation}{section}

\newcommand{\be}{\begin{equation}}
\newcommand{\ee}{\end{equation}}

\newcommand{\Red}[1]{\begingroup\color{red} #1\endgroup}

\begin{document}
\title[Multi--localized time-symmetric initial data]{Multi--localized time-symmetric initial data for the Einstein vacuum equations}

\author{John Anderson, Justin Corvino, Federico Pasqualotto}
\address{Department of Mathematics, Stanford University, Stanford, CA 94305, USA}
\email{jrlander@stanford.edu}
\address{Department of Mathematics, Lafayette College, Easton, PA 18042, USA}
\email{corvinoj@lafayette.edu}
\address{Department of Mathematics, UC Berkeley, Berkeley, CA 94720, USA}
\email{fpasqualotto@berkeley.edu}


\begin{abstract}
We construct a class of time-symmetric initial data sets for the Einstein vacuum equation modeling elementary configurations of multiple ``almost isolated" systems.  Each such initial data set consists of a collection of several localized sources of gravitational radiation, and lies in a family of data sets which is closed under scaling out the distances between the systems by arbitrarily large amounts. This class contains data sets which are not asymptotically flat, but to which nonetheless a finite ADM mass can be ascribed. The construction proceeds by a gluing scheme using the Brill--Lindquist metric as a template. Such initial data are motivated in part by a desire to understand the dynamical interaction of distant systems in the context of general relativity.  As a by-product of the construction, we produce complete, scalar-flat initial data with trivial topology and infinitely many minimal spheres, as well as initial data with infinitely many Einstein--Rosen bridges.
\end{abstract}

\maketitle

\tableofcontents

\section{Introduction} 
In this paper, we construct families of initial data sets for the Einstein vacuum equation. These data sets represent configurations of almost isolated and localized sources of gravitational radiation. The data we construct are motivated in part by studying the \emph{asymptotic stability} of Minkowski spacetime relative to initial data perturbations living in these classes (see Section~\ref{sec:motivation}).

We shall focus on constructing \emph{time-symmetric} initial data sets $(M,g)$, in which case the vacuum constraint equations reduce to the condition that the metric $g$ have vanishing scalar curvature. The constructions we consider are based on \emph{Brill--Lindquist metrics} (see \cite{bl}, cf. \cite{Mis63, Sor-Stav}). These metrics are constructed by considering a (finite or countably infinite) sequence $\mathfrak p$ of distinct points $p_i$ in $\R^3$ and a corresponding sequence $\mathfrak m$ of positive numbers $m_i$ meant to represent masses. The Brill--Lindquist metric associated to these points and masses is then
\[
(g^{\mathfrak m, \mathfrak p}_{BL})_{ij}(x)=\left (1 + \sum_k \frac{m_k}{2 |x-p_k|} \right )^4 \delta_{i j},
\]
where $|x-p_k|$ denotes the Euclidean distance from the point $p_k$. We assume the series converges uniformly on compact subsets of $\mathbb R^3\setminus \mathfrak p$, so that it in fact defines a positive harmonic function here.  The Brill--Lindquist metric is conformal to the Euclidean (flat) metric, and as the formula for the scalar curvature under a conformal change gives $R(g)= -8u^{-5}\, \Delta u$ for $g_{ij}=u^4 \delta_{ij}$, we have that the Brill--Lindquist metric has vanishing scalar curvature on $\mathbb R^3\setminus \mathfrak p$.

The main goal of this paper is to modify this metric locally near each of the points $p_i$ to yield a smooth metric with vanishing scalar curvature on $\mathbb R^3$, and we employ a gluing construction to accomplish this. Because we are interested in perturbations of the trivial initial data set $(\mathbb R^3, g_{ij}=\delta_{ij})$, the masses $m_i$ will be taken to be small, and for a given $k\in \mathbb Z_+$, the modifications near the points $p_i$ will ensure that the new metric is $C^k$-close to the Euclidean metric. It is important to remark that given a suitable mass sequence $\mathfrak m$, we do not need to impose any symmetry requirements on the sequence of points $\mathfrak p$, only that they are at a sufficiently large distance from each other.  Thus, we are constructing classes of $N$-body data sets, where the number $N$ is at most countable, and which can be construed as representing a sum of roughly spherically symmetric perturbations which are localized near distinct points.

In practice, we will first prove a basic gluing statement in the vicinity of a Schwarzschild metric (Theorem~\ref{thm:mainlemma}). We are then going to use this basic gluing construction to build examples in Theorem~\ref{thm:main} of time-symmetric vacuum initial data sets on $\mathbb R^3$ obtained by filling in suitably chosen Brill--Lindquist metrics $g^{\mathfrak m, \mathfrak c}_{BL}$ with what we construe as ``localized perturbations" of small mass $m_i$ near $c_i$.  The theorem focuses on $\mathfrak m$ countably infinite, whereas the construction works in the finite case, as noted in Corollary ~\ref{cor:N}.  The construction applies for $\|\mathfrak m\|_{\ell^\infty}$ suitably small, and so the results include data sets as follows: 

\begin{enumerate}
\item  We will construct initial data sets with an arbitrary finite number $N$ of localized perturbations.  In this case, each such initial data set is \emph{asymptotically flat} (albeit non-uniformly in the separation between the masses), with ADM mass $\sum^N_{i=1} m_i$. By increasing the number of localized pieces, we can achieve arbitrarily large ADM mass (see Remark~\ref{rmk:mass}).

\item  If $\mathfrak m\in \ell^1$ with $\|\mathfrak m\|_{\ell^\infty}$ is sufficiently small, the constuction will yield data with a countably infinite number of localized perturbations whose mass sequence $\mathfrak m= (m_i)_{i\in \mathbb Z_+}$, with each $m_i>0$,  has finite sum $\sum_{i = 1}^\infty m_i < \infty$.  There are data sets in this class with arbitrarily large $\|\mathfrak m\|_{\ell^1}$.

\item If $\mathfrak m\in \ell^\infty\setminus \ell^1$ with $\|\mathfrak m\|_{\ell^\infty}$ sufficiently small, the construction will produce data with $\|\mathfrak m\|_{\ell^1}=\infty$.  This includes cases in which $\sum_{i = 1}^\infty m_i = \infty$, but just barely so, in the sense that $\sum_{i = 1}^\infty \frac{m_i} {i^\delta} < \infty$ for some small $\delta > 0$.  In this case the resulting data is not asymptotically flat (see Section \ref{sec:totmass}).   
\end{enumerate}

In the second case, it is not clear from our construction whether any such data sets are asymptotically flat, though certainly some data sets in this class are \emph{not} asymptotically flat. As we will see, to make the data asymptotically flat would require choosing the masses for perturbations centered at large $|c_i|$ correspondingly small; that in turn affects a smallness condition in Theorem ~\ref{thm:mainlemma}, which may then require the centers to be moved further apart, and it is thus not clear whether the argument can close.  In any case, in Section \ref{sec:totmass}, we argue that $\| \mathfrak m\|_{\ell^1}$ may provide a suitable notion of mass for the data sets we construct.

These constructions are perhaps intriguing because they give initial data sets which are (i) far from symmetric, (ii) not necessarily asymptotically flat, (iii) have arbitrarily large (or even infinite) $\|\mathfrak m\|_{\ell^1}$, but which nonetheless still seem reasonable as perturbations relative to which Minkowski spacetime will be stable as a solution to the Einstein vacuum equations. Indeed, the motivation for these constructions comes from the study of asymptotic stability for the initial value problem, as we describe in Section~\ref{sec:motivation}.  As a consequence of the construction, we also show the existence of complete metrics with vanishing scalar curvature and with infinitely many minimal spheres.

We now provide precise statements of our results.  The constructions extend \emph{mutatis mutandis} to all $n\geq 3$, but for clarity of exposition, we give the statements and proofs for the case $n=3$, as the modifications for higher dimensions should be readily apparent.

\subsection{Statements}\label{sec:statements}

For $m > 0$, we let $g_S^m$ be the Schwarzschild metric in isotropic coordinates, defined on $\mathbb R^3\setminus\{0\}$ as follows:
\begin{equation}\label{eq:schwdef}
(g_S^m)_{ij}(x) =  \Big(1 + \frac{m}{2|x|} \Big)^4 \delta_{ij}.
\end{equation}

We will work locally around each point in $\mathfrak p$, using the following general gluing statement.

\begin{thm}\label{thm:mainlemma}
Let $\Omega= \{x \in \R^3: 1 < |x| < 2\}$, $\Omega_{\mathrm{int}} := \{x \in \R^3: 1< |x| < 9/8\}$, $\Omega_{\mathrm{ext}} := \{x \in \R^3: 15/8 < |x| < 2\}$. Given $2\leq N_0 \in \mathbb{Z}_{+}$ and $m>0$, and given $\varepsilon>0$, there exists a $\delta = \delta(m, \varepsilon, N_0)$ such that the following holds true. Consider a smooth metric $g_{\mathrm{ext}}$ defined on $ \Omega$, with the following properties:
\begin{enumerate}
\item[(i)] $g_{\mathrm{ext}}$ has vanishing scalar curvature,
\item[(ii)]  for all $x \in \Omega$, and all multi-indices $I$ with $0 \leq |I|\leq N_0+3$, we have \begin{align}|\partial^I_x((g_{\mathrm{ext}})_{ij}(x) - (g_S^{m})_{ij}(x))|<  \delta. \label{eq:close-schw-1} \end{align}
\end{enumerate}
Then there exists a metric $\widehat g \in \mathcal{C}^{N_0}( \Omega)$ and a number $\widehat m$, with $|m - \widehat m| < \min(m/2, \varepsilon)$, such that the following properties hold true: 
\begin{enumerate}
\item $\widehat g$ has vanishing scalar curvature,
\item $\widehat g = g^{\widehat m}_S$ for all $x \in  \Omega_{\mathrm{int}}$,
\item $\widehat g = g_{\mathrm{ext}}$ for all $x \in  \Omega_{\mathrm{ext}}$,
\item  for all $x \in \Omega$, and all multi-indices $I$ with $0 \leq |I|\leq N_0$, we have \begin{align}|\partial^I_x({\widehat g}_{ij}(x) - (g_S^{\widehat m})_{ij}(x))|< \varepsilon.  \label{eq:close-schw-2} \end{align}
\end{enumerate}
\end{thm}

\begin{remark} We can clearly replace $g_S^{\widehat m}$ in the preceding estimate with $g_S^m$ or $g_{\mathrm{ext}}$, since all these metrics can be made close. Furthermore, as we will recall later, the required assumption (\ref{eq:close-schw-1}) can be stated as a smallness condition on $(g_{\mathrm{ext}} - g_S^{m})$ in $C^{N_0+2, \alpha}(\Omega)$, and the resulting estimate (\ref{eq:close-schw-2} on $(\widehat g - g_S^{\widehat m})$ can be given in $C^{N_0,\alpha}(\Omega)$.  For related extension statements for initial data sets, compare \cite[Section 8.6]{cd}.
\end{remark}

As an application, we obtain the following statement concerning the existence of initial data sets with an arbitrarily large number of localized pieces.

\begin{thm}\label{thm:main}
Let $N_0\in \mathbb Z_+$ and $\varepsilon>0$. There exists $C>0$ and $0<\varepsilon_0< C^{-1} \varepsilon$ such that the following holds. For any sequence $\mathfrak{m}:=(m_i)_{i \in \mathbb{Z}_+}$ such that for all $i \in \mathbb{Z}_+$, $0<m_i \leq \varepsilon_0$ , there is a sequence $\mathfrak p=(p_i)_{i \in \mathbb{Z}_+}$ of distinct points in $\mathbb R^3$, so that, for all $\lambda\geq 1$, there exists a complete smooth metric $g$ on $\R^3$ with vanishing scalar curvature and the following properties: for all $x\in \R^3 \setminus \bigcup_{k\in \mathbb Z_+} \{x :|x- \lambda p_k| \leq 2\}$, 
\begin{align}
    g_{ij}(x) =  \Big(1 + \sum_{\ell\in \mathbb Z_+} \frac{m_\ell}{2|x-\lambda p_\ell|} \Big)^4\delta_{ij}, \label{eq:BLl}
\end{align}
while for all $k\in \mathbb Z_+$, and for all $x$ with $ |x-\lambda p_k| \leq 2$, and for all multi-indices $I$ with $0 \leq |I|\leq N_0$, $|\partial^I_x(g_{ij}(x) - \delta_{ij})|\leq Cm_k<\varepsilon$.
\end{thm}

We make a few remarks about this statement.
\begin{itemize}
    \item This theorem states that we can construct data sets which look like localized perturbations around the points $c_k=\lambda p_k$.
    \item The dependence of the points $\mathfrak p$ on the masses $\mathfrak m$ derives from two considerations: (i) to ensure that the series defining the conformal factor will converge uniformly, and (ii) so that around each $\lambda p_k$, the metric in (\ref{eq:BLl}) is an admissible $g_{\mathrm{ext}}$ to which we can apply Theorem~\ref{thm:mainlemma}.
    \item The parameter $\lambda$ in the statement reflects the fact that set of admissible choices of $\mathfrak p$ is closed under a rescaling of the location of the points by a factor $\lambda \geq 1$. This is important for future applications to the dynamics arising from these data sets because we want uniform estimates as the masses get farther apart. 
\end{itemize}

From the proof of Theorem~\ref{thm:main} and the calculations in Section~\ref{sec:BLest}, the following corollary will follow, yielding time-symmetric data containing a finite number $N$ localized perturbations about specified centers.

\begin{cor}\label{cor:N}
Let $N_0\in \mathbb Z_+$. For any $\varepsilon > 0$, there is a $\mu>0$ such that for any $N\in \mathbb Z_+$ and any positive numbers $m_1, \ldots, m_N$ with $\max_{i\leq N} m_i \leq \mu$, there is a $\sigma>0$ so that for any $c_1, \ldots, c_N\in \R^3$ with $\min_{i \neq j} |c_i - c_j| \geq \sigma$, there is a smooth metric $g$ with vanishing scalar curvature on $\R^3$ with the following properties: on $\R^3 \setminus \bigcup_{j\leq N} \{x :|x- c_j| \leq 2\}$, $g_{ij}(x) =  \Big(1 + \sum_{k\leq N} \frac{m_k}{2|x-c_k|} \Big)^4\delta_{ij}$, while for all $x$ in the set $\bigcup_{j\leq N} \{x : |x-c_j| \leq 2\}$, and all multi-indices $I$ with $0 \leq |I|\leq N_0$, $|\partial^I_x(g_{ij}(x) - \delta_{ij})|< \varepsilon$.
\end{cor}

\begin{remark} We emphasize that we can fix masses $m_1, \ldots, m_N$ each of size at most $\mu$, and then the construction works with those same masses, for all choices of centers $c_1, \ldots , c_N$ with $\min_{i \neq j} |c_i - c_j|$ sufficiently large: each localized perturbation retains the same ``strength" even if they are separated by a large distance.
\end{remark}

 \begin{remark} By letting $\frac{\mu}{N^\delta}\leq m_i\leq \mu$ for $1\leq i\leq N$ and for some $0<\delta <1$, for example, it follows from the corollary that there exist such configurations with $\sum_{i\leq N} m_i$ arbitrarily large.   \label{rmk:mass}
\end{remark} 

The following is a direct corollary of the estimates in Section~\ref{sec:BLest} along with Theorem~\ref{thm:mainlemma}.  The Riemannian manifold $(\widehat M, \widehat g)$ can be construed as vacuum initial data with infinitely many Einstein--Rosen bridges (see also \cite{Mis63} where an arbitrary finite number of Einstein--Rosen bridges are constructed). 

\begin{cor} For any sequence $\mathfrak m$ with each $m_i>0$, there is a sequence $\mathfrak p$ as above such that for any $\lambda \geq 1$ and with  $\mathfrak c=\lambda \mathfrak p$, there is a complete smooth metric $g$ on $\mathbb R^3\setminus  \mathfrak c$ with vanishing scalar curvature, with $g= g_{BL}^{\mathfrak m, \mathfrak c}$ outside $\bigcup_{j\in \mathbb Z_+} \{x :|x-  c_j| \leq 2\}$, while for $0<|x_k-c_k|\leq 1$, $g_{ij}(x)= \Big(1 + \frac{\widehat m_k}{2|x-c_k|} \Big)^4 \delta_{ij}=: (g_S^{\widehat m_k, p_k})_{ij}(x)$, where $m_k/2 \leq \widehat{m}_k \leq  3 m_k/2$ for all $k \in \mathbb{Z}_+$.  If $m_k\leq 4/3$, then the Schwarzschild end around $c_k$ includes the minimal sphere (horizon) $\Sigma_k:= \{x:|x-c_k|= \tfrac{\widehat m_k}{2}\}$. 

 By doubling $(\mathbb R^3 \setminus \bigcup_{k\in \mathbb Z_+} \{ x:|x-c_k|\leq \tfrac{\widehat m_k}{2}\}, g)$ over the totally geodesic boundary $\Sigma:=\bigcup_{k\in \mathbb Z_+} \Sigma_k$, we obtain a complete smooth Riemannian manifold $(\widehat M, \widehat g)$ with vanishing scalar curvature.
\end{cor}

The manifolds in this corollary have infinite topology, whereas with a bit more work, which we carry out in Section~\ref{sec:multi-h}, we obtain the following result. 

\begin{prop} \label{prop:multi-h} There are complete smooth metrics on $\mathbb R^3$ with vanishing scalar curvature and infinitely many minimal spheres. 
\end{prop}

\subsection{Motivation}\label{sec:motivation}
In the context of the Einstein vacuum equation, an initial data set consists of a three-dimensional Riemannian manifold $(M,g)$ and a symmetric two-tensor field $K$ along $M$. Because $K$ will end up being the second fundamental form of an embedding of $(M,g)$ into a vacuum spacetime, the Gauss and Codazzi equations tell us that $g$ and $K$ cannot be chosen freely, but must satisfy certain compatibility conditions, the Einstein constraint equations. Under these constraints, the work of Choquet-Bruhat and Geroch (see \cite{CB52} and \cite{ChBrGe69}) guarantees the existence of a $3 + 1$--dimensional Lorentzian manifold $(\overline{M},\overline{g})$ such that $(M,g)$ embeds isometrically in $\overline{M}$ as a Cauchy hypersurface with $K$ as its second fundamental form (see \cite{Ri09}). The properties of solutions of the Einstein constraint equations, including techniques such as the \emph{conformal method} to construct solutions, have been the focus of intense study in mathematical physics and geometric analysis (see, for example, \cite{bi-ce,gms-ece} and references therein).

The class of data we construct is motivated by a desire to study the stability of the Minkowski spacetime (the trivial solution of the Einstein vacuum equation) relative to perturbations which are localized near several points separated by large distances. In the monumental work \cite{ChrKla93}, Christodoulou--Klainerman showed that, relative to sufficiently small perturbations localized near a single point, the trivial solution is globally asymptotically stable. This means that the vacuum spacetimes arising from all such initial data sets do not form black holes, and in fact asymptotically converge back to flat Minkowski spacetime in an appropriate sense by radiating away gravitational energy. The goal is to show that something similar happens in the context of these perturbations localized near several points, and in fact to show that the perturbations ``decouple" from each other as they are taken farther and farther apart. This would allow us to probe, in the simplest setting, the interaction between several nearly isolated systems in the context of general relativity. These kinds of results also bring us closer to proving asymptotic stability for quasilinear equations in translation-invariant function spaces.

For the model case of systems of quasilinear wave equations satisfying Klainerman's \cite{Kla82} \emph{null condition} (which models nonlinear structure found in the Einstein vacuum equation in an appropriate gauge), these kinds of statements were shown in~\cite{AndPas22}. The data there were comprised of a sum of roughly spherically symmetric functions supported near $N$ points. The main analytical difficulty arises from the fact that norms measuring the size of perturbations usually have radial weights away from some chosen center. These norms grow polynomially in the distance between the points, and thus, it is the large distance limit which we are interested in understanding. Moreover, since the analysis focused on systems of nonlinear wave equations, there was no need to construct appropriate classes of initial data. The data from~\cite{AndPas22} corresponds closely to one of the classes of data from Theorem~\ref{thm:main}. Thus, the main results of this paper construct data which are somewhat analogous to the ones considered in~\cite{AndPas22}, and also provide some generalizations which may be interesting to consider in the context of asymptotic stability.

\subsection{$N$-body initial data sets} 

We now compare the classes of data sets we construct with other $N$-body initial data set constructions. We focus on versions of statements involving a single end and vacuum data, as this is the case most closely related to the present work.

In \cite{cd:as-pen,cd}, Chru\'{s}ciel--Delay construct initial data sets containing any finite number of disconnected regions, in each of which the data agrees with that from a suitably chosen Schwarzschild or Kerr black hole. The work \cite{ptc-maz} of Chru\'{s}ciel--Mazzeo then proves that in the vacuum spacetime development of suitable members of the constructed family of initial data, the intersection of certain Cauchy hypersurfaces with the complement of the past of future null infinity has multiple connected components, i.e. these developments represent multiple black hole spacetimes.  The works \cite{cci-ts,cci} of Chru\'{s}ciel--Corvino--Isenberg employ similar gluing constructions to produce vacuum initial data sets with disconnected regions from any finite number of solutions of the constraints, allowing more general configurations than those above.  In all these constructions, outside of a compact set, the solutions are given by initial data for a Kerr metric (reducing to Schwarzschild when time-symmetric).

A strikingly different configuration is produced in the work~\cite{car-sch} of Carlotto--Schoen, which constructs $N$-body data sets glued inside any finite number of conical regions in $\R^3$. Outside of these cones, the metrics can be flat. Because of finite speed of propagation, these gravitational disturbances will not interact for some (potentially large) time, while the spacetime will contain a (potentially large) region isometric to a region in Minkowski spacetime. These kinds of data sets have also been constructed recently using other techniques. The work~\cite{aretakis-czimek-rodnianski} of Aretakis--Czimek--Rodnianski employs characteristic gluing, while the work \cite{mao-tao} of Mao--Tao constructs particular fundamental solutions to the linearized scalar curvature equation.

In contrast, the data sets constructed here are cooked up to look like sums of roughly spherically symmetric pieces centered at different points. In the spirit of ~\cite{AndPas22}, we are also interested in having constructions that are uniform as the separation between the points becomes infinite. Thus, the Brill--Lindquist metrics are natural template candidates to obtain the desired properties.

We now briefly comment on the proof strategy for the main gluing constructions in Theorem~\ref{thm:mainlemma} and Theorem~\ref{thm:main}.
\subsection{Description of the proof}
We use the kind of gluing construction originating in \cite{cor:schw} (see also~\cite{cs:ak, cd} for the non-time-symmetric case). 

We first focus on the proof of Theorem~\ref{thm:mainlemma}.  We consider a smooth combination of $g_{\text{ext}}$ and $g^{\widetilde m}_{S}$: 
$$
\widetilde g := (1-\psi) g^{\widetilde m}_S + \psi g_{\text{ext}},
$$ 
where $\psi$ is a radial function which transitions from $0$ to $1$ in an annulus compactly contained in $\Omega= \{x \in \R^3: 1 < |x| < 2\}$, and we think of $\widetilde m$ as close to $m$. 

The core of the argument is a Lyapunov--Schmidt reduction. We linearize the scalar curvature of $\widetilde g$ around $g^{m}_{S}$, and we consider the kernel of its formal adjoint $L_{g_S^{m}}^*$. This kernel is one-dimensional, spanned by a radial function $f^{m}$. It then follows (Proposition~\ref{prop:def} that there exists a smooth tensor $h(\widetilde g)$ supported on $\overline \Omega$ for which $\widetilde g + h(\widetilde g)$ is a metric with 
\begin{equation}\label{eq:intro1}
R\big(\widetilde g + h(\widetilde g)\big) = \eta f^{m}
\end{equation}
where $\eta$ is a multiple of a bump function.  Moreover the map $g \mapsto h(g)$ is well-behaved, in the sense that it is continuous as a function of $g$, and $h(g)$ satisfies appropriate H\"older and Sobolev estimates.

As $\eta$ depends on $\widetilde m$, the final step is to show, by a continuity argument, that $\widetilde m$ can be modulated to a suitable $\widehat m$ to achieve $\eta = 0$ in~\eqref{eq:intro1}.  This shows that we can ``fill in'' the metric $g_{\text{ext}}$ with an exact Schwarzschild initial data set.

Concerning the proof of Theorem~\ref{thm:main} we use Theorem~\ref{thm:mainlemma} setting $g_{\text{ext}}$ equal to a Brill--Lindquist metric:
$$
g_{\text{ext}} = g_{BL}^{\mathfrak m, \lambda \mathfrak p} := \Big(1 + \sum_k \frac{m_k}{2|x - \lambda p_k|} \Big)^4 \delta_{ij}.
$$
For suitable such metrics, Theorem~\ref{thm:mainlemma} allows us to glue a piece of Schwarzschild in each annulus centered at $\lambda p_k$. To complete the construction, we need to cap this off with an ``inner'' region.  These regular ``inner'' pieces of data will be constructed by rescaling the data sets constructed in~\cite{cd:as-pen, cor:as-pen} (see Proposition \ref{prop:sm-mass}) to exactly match the outer asymptotics given by the mass $\widehat m$ in the transition region.

We use the fact that the construction is done locally in each of the gluing regions. We thus need to take care so that the configuration allows one to apply the gluing construction to each individual annulus. This can be ensured by possibly enlarging the separation distance between each of the localized pieces of initial data in order to satisfy the hypotheses of Theorem~\ref{thm:mainlemma}.

\subsection{Outline of the paper}
In Section~\ref{sec:MainLem} we provide the proof of Theorem~\ref{thm:mainlemma} (with additional details found in Appendix~\ref{sec:proof-prop}). In Section~\ref{sec:BLest}, we prove some estimates on the Brill--Lindquist metric which will be necessary for appropriately applying the gluing construction from Theorem~\ref{thm:mainlemma} in order to prove Theorem~\ref{thm:main}. In Section~\ref{sec:ThmMain}, we provide the proof of Theorem~\ref{thm:main}. In Section~\ref{sec:multi-h}, we provide the necessary modifications to prove Proposition~\ref{prop:multi-h}. Finally, in Section~\ref{sec:totmass}, we discuss a notion of ADM mass for some of the data sets we construct.

\subsection{Acknowledgements}
JA gratefully acknowledges that this work was partly supported by the National Science Foundation under Grant No. 2103266.

\section{Proof of Theorem \ref{thm:mainlemma} } \label{sec:MainLem}

\subsection{Preliminaries}
For $m>0$, we consider the Riemannian Schwarzschild metric in centered isotropic coordinates: $(g^m_S)_{ij}(x)=(1+ \tfrac{m}{2|x|})^4 \delta_{ij}$ on $\mathbb R^3\setminus\{0\}$.  The portion for $|x|>\tfrac{m}{2}$ is isometric to the metric $(1-\tfrac{2m}{r})^{-1} dr^2 + r^2 \mathring g_{\mathbb S^2}$ on the space $\{ r: r>2m\} \times \mathbb S^2$, where $\mathring g_{\mathbb S^2}$ is the round unit sphere metric. 

Let $L_g$ be the linearized scalar curvature operator, given by the well-known formula $L_g(h)= -\Delta_g (\mathrm{tr}_gh) + \mathrm{div}_g \mathrm{div}_g h - \langle h,  \mathrm{Ric}(g)\rangle_g$, with formal $L^2$ adjoint $L_g^*f= -(\Delta_g f)g+ \mathrm{Hess}_g f - f \mathrm{Ric}(g)$.   By explicit calculation, for metrics $\gamma$ and $\gamma +h$, $R(\gamma+h) = R(\gamma) + L_\gamma (h) + \mathcal Q_{\gamma}(h)$, where $\mathcal Q_{\gamma}(h)$ is expressed in coordinates as a homogeneous quadratic polynomial in $h_{ij}$, $h_{ij,k}$ and $h_{ij,k\ell}$, with coefficients that depend smoothly on $\gamma_{ij}$, $\gamma_{ij,k}$, $\gamma_{ij,k\ell}$, as well as $h_{ij}$, $h_{ij,k}$ and $h_{ij,k\ell}$.

Since $g_S^m$ has vanishing scalar curvature, any element $f$ in the kernel of $L^*_{g_S^m}$ satisfies $\mathrm{Hess}_{g_S^m}(f) = f \mathrm{Ric}(g_S^m)$ (this follows from taking the trace of the equation).  A direct computation proves that a kernel element $f$ is a radial function (i.e. it only depends on $|x|$), and in fact the kernel of $L^*_{g_S^m}$ is spanned by, in isotropic coordinates, the function $f^m(x)= \frac{ 1-\frac{m}{2|x|}}{1+\frac{m}{2|x|}}$. If we change to coordinates for which $g_S^{m}= (1-\tfrac{2m}{r})^{-1} dr^2 + r^2 \mathring g_{\mathbb S^2}$, then the kernel is spanned by $\sqrt{1-\tfrac{2m}{r}}$. Indeed, under this change of coordinates, $2|x|= r-m+\sqrt{r(r-2m)}$, from which we conclude.

For a metric $g$, we will denote the volume measure by $d\mu_g$, and induced hypersurface measure by $d\sigma_g$; for the Euclidean metric, we may use to $dx$ and $dA$, respectively.

\subsection{The nonlinear projected problem}

Let $B= \{x\in \mathbb R^3: |x|< 2\}$, and recall $\Omega= \{x\in \mathbb R^3:1< |x|< 2\}$.

Let $d(x)= d(x, \partial \Omega)$: for $x\in \Omega$, $d(x)= \min (|x-1|, |x-2|)$.   Let $ \mathring \zeta:\mathbb R\rightarrow [0,1]$ be smooth and non-decreasing, $\mathring \zeta(t) = 0 $ for $t\leq \tfrac18$, say, while $\mathring \zeta(t) = 1$ for $t\geq \tfrac14$. Define 
$$
\zeta(x)= \begin{cases} \mathring \zeta(d(x)), \;\quad x\in \Omega \\ 0 , \qquad \qquad x\in \mathbb R^3\setminus \Omega \end{cases}.
$$
We let $\mathring \psi:\mathbb R\rightarrow \mathbb R$ be a smooth nondecreasing function with $\mathring \psi(t)=0$ for $t\leq \tfrac54$, while $\mathring \psi(t)=1$ for $t\geq \frac74$.
Let $\psi:\mathbb R^3\rightarrow \mathbb R$ be defined as $\psi(x) = \begin{cases} \mathring \psi(|x|), \quad  |x|< 2 \\  1, \;\quad\qquad  x\in \mathbb R^3\setminus  B\end{cases}.$

 As a first step to proving Theorem \ref{thm:mainlemma}, we show we can perturb $\widetilde g$ on $\Omega$ to a metric with vanishing scalar curvature modulo a one-dimensional error.   The H\"{o}lder norms in the next proposition can be computed at a (smooth) background metric, say, or at $\gamma$, and a similar remark holds for the proof of the theorem below. 

\begin{prop} Let $0<\alpha<1$ and $\Omega' := \{x \in \R^3: 9/8 < |x| < 15/8\}\subset \Omega$.   There is a constant $C_0>0$ such that for (smooth) metrics $\gamma$ sufficiently near $g^m_S$ in $C^{4,\alpha}(\overline{\Omega})$ and with scalar curvature $R(\gamma)$ supported in $\overline{\Omega'}$, there exist a constant $ b(\gamma)$ along with a smooth symmetric tensor $h(\gamma)$, which extends smoothly by $0$ outside $\Omega$, such that $\gamma+h(\gamma)$ is a metric with $$R(\gamma+h(\gamma))= b(\gamma) \zeta f^{m} .$$   Moreover, $\gamma \mapsto (h(\gamma), b(\gamma)):C^{4, \alpha}(\overline{\Omega})\rightarrow C^{2, \alpha}(\overline{\Omega})\times \mathbb R$ is continuous, and we have the estimate $\|h(\gamma)\|_{C^{2,\alpha}} \leq C_0 \|R(\gamma)\|_{C^{0, \alpha}}$.  In fact for $k\in \mathbb Z_+$, there is $C_k>0$ such that $ \|h(\gamma)\|_{C^{k+2,\alpha}} \leq C_k \|R(\gamma)\|_{C^{k, \alpha}}$. \label{prop:def} \end{prop}

We postpone the proof, which follows along the lines of proofs of related results in \cite{cor:schw, cs:ak, cor-lan}, cf. \cite{cd}, to Appendix~\ref{sec:proof-prop}, where we will prove a more general statement, Proposition~\ref{prop:def0}.  With this fact in hand, we are ready to prove Theorem \ref{thm:mainlemma}.

\subsection{Proof of Theorem~\ref{thm:mainlemma}}
\begin{proof} Let $\varepsilon >0$.  For $\widetilde m>0$, we define 
$$
\widetilde g_{ij}(x) = (1- \psi(x))(g^{\widetilde m}_S)_{ij}(x) + \psi(x)(g_{\text{ext}})_{ij}(x).
$$
There is a $C^{4, \alpha}(\overline{\Omega})$-neighborhood $\mathcal U$ of $g_S^m$ and a $0<\xi<\min(\tfrac{m}{2},\varepsilon, 1)$, such that for $J_\xi=[m-\xi, m+\xi]$ and for any $g_{\mathrm{ext}}\in \mathcal U$ and $\widetilde m\in J_\xi$, $\widetilde g$ as above is sufficiently near $g_S^m$ to apply Proposition~\ref{prop:def} with $\gamma=\widetilde g$ to solve for $h(\tilde g)$, satisfying the conditions therein. Moreover there is a $\delta_0>0$ such that for metrics $g_{\mathrm{ext}}$ satisfying (\ref{eq:close-schw-1}) with $\delta< \delta_0$, $g_{\mathrm{ext}}\in \mathcal U$.  The constants in the estimates below are independent of $\widetilde m\in J_\xi$. 

We then define $\widehat h$ to be the difference $\widehat h := g_{\text{ext}} - g^m_S$, so that 
$$
\widetilde g= (1-\psi) g_S^{\widetilde m} + \psi g_S^{m}  + \psi \widehat h =:g_{\psi} + \psi \widehat h.
$$  
Now, $R(\widetilde g+h(\widetilde g))$ vanishes precisely when $$\int\limits_{\Omega} f^{m} R(\widetilde g+h(\widetilde g))\; d \mu_{g^m_S}=0.
$$
We note that
\begin{align*} 
R(\widetilde g+h(\widetilde g))&= R( (1-\psi) g_S^{\widetilde m} + \psi g_S^{m}  )+ L_{g_{\psi}}( \psi \widehat h+h(\widetilde g))+\mathcal Q_{g_{\psi}} (\psi \widehat h + h(\widetilde g))\\
&= R( (1-\psi) g_S^{\widetilde m} + \psi g_S^{m} )+L_{g_S^m}(\psi \widehat h+h(\widetilde g))\\
& \qquad  +(L_{g_{\psi}} - L_{g_S^m})(\psi \widehat h + h(\widetilde g))+ \mathcal Q_{g_{\psi}}(\psi \widehat h+h(\widetilde g)).
\end{align*}

The integral of the term $f^m L_{g^m_S}(h(\widetilde g))$ vanishes because $f^m$ is in the kernel of $L^*_{g^m_S}$, and $h(\widetilde g)$ and its derivatives vanish along the boundary.

We can estimate the term containing $L_{g^m_S}(\psi \widehat h)$: for some $C>0$,
\begin{equation*}
    \left| \int\limits_{\Omega} f^{m} L_{g_S^{m}}(\psi \widehat h)d\mu_{g^m_S} \right| \leq C \|\widehat h\|_{C^2(\Omega)}.
\end{equation*}

Moreover, there is a constant $C>0$ such that on $\overline{\Omega}$ 
$$|\mathcal Q_{g_{\psi}}(\psi \widehat h+h(\widetilde g))|\leq C (\|\widehat h\|^2_{C^2(\Omega)}+ \|h(\widetilde g)\|^2_{C^2(\Omega)}),$$ and
$$
|( L_{g_{\psi}}-L_{g_S^{m}}) (\psi\widehat h+h(\widetilde g))|\leq  C\|g^{m}_S-g^{\widetilde m}_S\|_{C^2(\Omega)} \cdot(\|\widehat h\|_{C^2(\Omega)}+ \|h(\widetilde g)\|_{C^2(\Omega)}).
$$

As for evaluating $\int\limits_{\Omega} f^{m} R( (1-\psi) g_S^{\widetilde m} + \psi g_S^{m}  )\; d\mu_{g^m_S}$, it is actually more convenient to change the radial coordinate to the area radius $\rho$, as shown in \cite{li-yu}.  In such coordinates, $f^{m}=\sqrt{ 1- \tfrac{2m}{\rho}}$, and so if we let $\mathring g_{\mathbb S^2}$ be the metric on the round unit sphere with area measure $d\mathring \sigma$, we have $g^{m}_S= (1-\tfrac{2m}{\rho})^{-1} d\rho^2 + \rho^2 \mathring g_{\mathbb S^2}$, and $f^{m}d \mu_{g^m_S}=\rho^2 d\rho \;d\mathring \sigma$.  If we let $F(\rho)= \Big((1-\psi) (1-\tfrac{2\widehat m}{\rho})^{-1}+ \psi (1-\tfrac{2m}{\rho})^{-1}\Big) $, we have $g_{\psi}= F(\rho) d\rho^2+\rho^2 \mathring g_{\mathbb S^2}$.  As in \cite[p. 763]{li-yu}, then $R(g_{\psi})=-2\rho^{-2} \frac{\partial}{\partial \rho} \Big( \rho (F(\rho))^{-1} -\rho\Big)$.  From here it is easy to compute 
\begin{align}
\int\limits_{\Omega} f^{m} R((1-\psi) g_S^{\widetilde m} + \psi g_S^{m})\; d \mu_{g^m_S}= 16\pi (m-\widetilde m).  \label{eq:mint}
\end{align}

With $0<\delta < \min(\delta_0, \xi^2)$, using (\ref{eq:close-schw-1}) and applying Proposition~\ref{prop:def}, we have 
$$
\|\widehat h\|_{C^{N_0+2,\alpha}(\Omega)} \leq  \delta, \qquad \|h(\widetilde g)\|_{C^{N_0,\alpha}(\Omega)} \leq C \|R(\widetilde g)\|_{C^{N_0-2,\alpha}} \leq C' (|m-\widetilde m|+\delta).
$$
Thus we conclude that there is a constant $C>0$ such that 
\begin{align*} 
\left|\int\limits_{\Omega} f^{m} R(\widetilde g+h(\widetilde g))\; d \mu_{g^m_S}- 16\pi(m-\widetilde m)\right| &\leq C ((m-\widetilde m )^2+  \delta).
\end{align*}
By Proposition \ref{prop:def}, $$\widetilde m\mapsto \mathcal I(\widetilde m):= \int\limits_{\Omega} f^{m} R(\tilde g+h(\tilde g))\; d \mu_{g^m_S}
$$ 
is continuous for $\widetilde m \in J_\xi$.  For $\xi$ sufficiently small, we then have by the above estimates that $\mathcal{I}(m-\xi) \cdot \mathcal{I} (m+\xi) < 0$. Thus by continuity, for some $\widehat m\in J_\xi$, $\mathcal I(\widehat m)= 0$, as desired.  \end{proof}


\section{Brill--Lindquist estimates}  \label{sec:BLest}

In order to prove Theorem~\ref{thm:main}, we show that for suitable $\mathfrak m$ and $\mathfrak p$, the series defining the metric $g_{B L}^{\mathfrak{m},\mathfrak{p}}$ is well defined on $\mathbb R^3\setminus \mathfrak p$, and that it satisfies the estimates necessary to apply Theorem~\ref{thm:mainlemma} near each point. We also want that the family of data sets we construct is closed under rescalings which move the points in $\mathfrak{p}$ away from each other. We thus begin with a discussion about sufficient conditions for both of these requirements to hold. In doing so, we shall see that the sufficient conditions we find are indeed closed under rescalings of the desired kind. We shall then construct points $\mathfrak{p}$ corresponding to an arbitrary sequence of masses $\mathfrak{m}$. The proof is rather straightforward, leveraging that near each $p_k$, the Brill--Lindquist metric is close to the corresponding Schwarzschild metric, with errors depending on the distance between the points in $\mathfrak{p}$.

We recall that $\mathfrak{m} = (m_i)_{i\in \mathbb Z_+}$, with each $m_i>0$, and $\mathfrak{p} = (p_i)_{i\in \mathbb Z_+}$, with $|p_i-p_j|\geq 5$ for $i\neq j$. We now assume that $\mathfrak{p}$ are chosen such that for all $x$ a distance at least $0<\theta \leq 1$ away from each $p_\ell$, we have that  $\sum_\ell \frac{m_\ell }{ 2 |x - p_\ell|} \le C\theta^{-1}(1+|x|)$. This is true whenever $\sum_\ell \frac{m_\ell}{1 + |p_\ell|}$ is finite: given any such admissible $x$, we consider the (finite) set $P_C$ of points $p_k\in \mathfrak p$ a distance at most $10 (1+|x|)$ from the origin, and let $P_F=\mathfrak p \setminus P_C$ be the set of remaining points.  We have that
\[
\sum_\ell {m_\ell \over 2 |x - p_\ell|} = \sum_{p_\ell \in P_C} {m_\ell \over 2 |x - p_\ell|} + \sum_{p_\ell \in P_F} {m_\ell \over 2 |x - p_\ell|}\, .
\]
The second sum is less than $\sum_\ell {m_\ell \over 1 + |p_\ell|}$, because for all $p_\ell \in P_F$, $\tfrac{11}{10}|p_\ell|\geq 1+|p_\ell|$, so that $|x-p_\ell|\geq |p_\ell|-|x|\geq \tfrac{9}{10} |p_\ell|\geq \tfrac{9}{11}(1+|p_\ell|)$. The first sum is less than $6(1 + |x|) \theta^{-1} \sum_\ell {m_\ell \over 1 + |p_\ell|}$, since for $p_\ell\in P_C$, $1+|p_\ell|< 12(1+|x|)$, giving us the desired result. We thus have a uniformly convergent series on every compact subset of $\R^3 \setminus \mathfrak{p}$, meaning that the series defines a harmonic function, and the metric $g_{B L}^{\mathfrak{m},\mathfrak{p}}$ thus has vanishing scalar curvature. We note that the derivatives of the series also converge uniformly. All of this remains true after rescaling the positions of the points by some $\lambda \ge 1$, since for such $\lambda$, 
\[
\sum_\ell {m_\ell \over 1 + \lambda |p_\ell|} \le \sum_\ell {m_\ell \over 1 + |p_\ell|}\, .
\]

With Theorem~\ref{thm:mainlemma} in mind, for positive integers $k$ and $N_1$, we estimate the difference $g_{B L}^{\mathfrak{m},\mathfrak{p}}-g_S^{m_k,p_k}$ in $C^{N_1}$ on the set $\overline{\Omega}_k$ of points $x$ with $1 \le |x - p_k| \le 2$, ensuring these metrics are close enough depending on $m_k$. To this end we write
\begin{equation} \label{eq:BLMetricExpansion}
\begin{aligned}
(g^{\mathfrak m, \mathfrak p}_{\mathrm{BL}})_{ij}(x) &= \left ( 1+ \sum_{\ell = 1} \tfrac{m_\ell}{2|x - p_\ell|} \right )^4 \delta_{ij}
\\ &= \left (1 + {m_k \over 2 |x - p_k|} \right )^4 \left (1 + \sum_{\ell \ne k} \Theta_{\ell k} (x) \right )^4 \delta_{i j},
\end{aligned}
\end{equation}
where for $\ell\neq k$, 
\begin{equation}\label{eq:kl}
\Theta_{\ell k} (x) = {m_\ell \over |x - p_\ell|} {|x - p_k| \over m_k + 2 |x - p_k|}.
\end{equation}
Thus, for any multi-index $I$ and for any $x$ with $1 \le |x - p_k| \le 2$, we have that 
\begin{equation} \label{eq:seriesbound1}
\begin{aligned}
|\partial^I_x \Theta_{\ell k} (x)| \le C m_\ell {1 \over |p_k - p_\ell|}, 
\end{aligned}
\end{equation}
where $C$ depends on $I$, but not on $\mathfrak m$ or $\mathfrak p$. We now define
\[
\Theta_k (x) = \sum_{\ell \ne k} \Theta_{\ell k} (x).
\]
We want to compare the Brill--Lindquist metric with a Schwarzschild metric of mass $m_k$ near each point $p_k$.  Given any sequence $\hat\delta_k>0$, we will pick the points $\mathfrak{p}$ such that for each $k$,
\begin{align}
\sum_{\ell \ne k} {m_\ell \over |p_k - p_\ell|} \leq \hat{\delta}_k. \label{eq:sum-est}
\end{align}
Given $\delta_k>0$, by picking $\hat{\delta}_k = \hat{\delta}_k (\delta_k)$ sufficiently small, we see by explicit expansion of the metric components in \eqref{eq:BLMetricExpansion} along with \eqref{eq:seriesbound1} that 
\[
\Vert (g_{B L}^{\mathfrak{m},\mathfrak{p}} - g_{S}^{m_k,p_k})_{i j} \Vert_{C^{N_1}(\overline{\Omega}_k)} < \delta_k.
\]
This condition is closed under the appropriate rescaling, as
\[
\sum_{\ell \ne k} {m_\ell \over \lambda |p_k - p_\ell|} \le \sum_{\ell \ne k} {m_\ell \over |p_k - p_\ell|}
\]
for $\lambda \ge 1$.  We will apply this by letting $\delta_k\leq \tfrac12 \delta$ from (\ref{eq:close-schw-1}) in Theorem~\ref{thm:mainlemma} for gluing near the metric $g_S^{m_k}$. While the $p_k$ have not yet been explicitly chosen, we note that $\delta_k$ depends only on $m_k$ and not on $p_k$.  Furthermore, as $|p_k-p_\ell|\leq |p_k|+|p_\ell|$, we have $$\sum_{\ell\neq k} {m_\ell \over |p_k| + |p_\ell|} \leq \sum_{\ell\neq k} {m_\ell \over |p_k-p_\ell|}\; ,$$ whence $\sum_\ell {m_\ell \over 1 + |p_\ell|}$ converges, as desired.

Given a sequence $\mathfrak{m}$ of positive masses, it is now relatively straightforward to construct a sequence of points $\mathfrak{p}$ which will result in a metric $g_{B L}^{\mathfrak{m},\mathfrak{p}}$ satisfying the required properties. One such construction is as follows.  

First, we define quantities to capture how much ``space'' we have in the required estimates, and ensure we continue to have room to add successive points and maintain the estimates: for $k\geq 1$ and $1\leq q\leq k$, let 
\begin{align*}
    \Xi_{q, k}&= \hat{\delta}_q - \sum_{\ell\neq q}^k {m_\ell \over |p_q - p_\ell|} . 
    \end{align*}
In case $q=k$ we let $\Xi_k= \hat{\delta}_k - \sum_{\ell \le k - 1} {m_\ell \over |p_k - p_\ell|}$.  Observe that for $1\leq q\leq k<k'$, $\Xi_{q, k}> \Xi_{q, k'}$. 

We let $p_1 = 0$, and note $\Xi_1=  \hat\delta_1$.  Let $p_2\neq p_1$ be chosen such that 
$$|p_2|=5+ 2m_2\Xi_1^{-1}+2\hat\delta_2^{-1}m_1.$$
Clearly $|p_2-p_1|=|p_2|\geq 5$, and
\begin{align*}
    \Xi_2&=  \hat\delta_2 - \frac{m_1}{|p_2-p_1|} = \hat\delta_2 - \frac{m_1}{|p_2|} \geq\frac12 \hat\delta_2\\
\Xi_{1, 2}&=  \hat\delta_1 -\frac{m_2}{|p_1-p_2|}= \hat\delta_1 - \frac{m_2}{|p_2|} \geq\frac12 \hat\delta_1.
    \end{align*}

We now proceed recursively to define $\mathfrak p$ such that for all $k$, $\Xi_k\geq \tfrac12 \hat\delta_k$, and $\Xi_{q,k}\geq 2^{-(k-1)} \hat \delta_q$ for all $q<k$.  These ensure that as we proceed in the recursion, the required estimates persist: as long as $\Xi_{k,k'}$ is nonnegative for all $k<k'$, (\ref{eq:sum-est}) will be satisfied.  As the recursion proceeds, we take care to choose $p_k$ so that it costs at most half of the remaining space for each of the $p_q$ with $q \le k - 1$. We will also arrange that for all $\ell \neq k$, $|p_k-p_\ell|\geq 5$.  In fact, in order to obtain the estimates, we arrange that $|p_q-p_\ell|\geq \tfrac12 |p_q|$ for all $\ell< q$.

For $k\geq 3$, suppose we have $p_1, \ldots, p_{k-1}$, with $\Xi_\ell\geq \tfrac12 \hat\delta_\ell$ for $1\leq \ell\leq k-1$, with $\Xi_{q, k-1}\geq 2^{-(k-2)} \hat\delta_q$ for $1\leq q< k-1$, and with $|p_q-p_\ell|\geq \max (5, \tfrac12|p_q|)$ for $\ell< q \leq k-1$.   Observe that these conditions do hold for $k=3$.

To proceed, for $k\geq 3$, we pick $p_k$ to be an arbitrary point with
\[
|p_k| = 2 |p_{k - 1}| + 4 m_k \max_{\ell \le k - 1} \Xi_{\ell, k-1}^{-1} + 4 \hat{\delta}_k^{-1} \sum_{\ell \le k - 1} m_\ell.
\]  
Each summand is designed to achieve a required estimate. First, we readily see 
$$|p_k-p_{k-1}|\geq |p_k|-|p_{k-1}|\geq |p_{k-1}|+ 4m_k \max_{\ell \le k - 1} \Xi_{\ell, k-1}^{-1} + 4 \hat{\delta}_k^{-1} \sum_{\ell \le k - 1} m_\ell\geq \tfrac12 |p_k|,$$ and in fact, since for $\ell< k$, $|p_k|\geq 2^{k-\ell} |p_\ell|$, we have
$$|p_k-p_\ell|\geq |p_{k}|-|p_\ell|\geq \tfrac12 |p_k|.$$ It also follows that $|p_k|\geq 10$ for $k\geq 3$, so that $|p_k-p_\ell|\geq 5$ for all $k\neq \ell$. 

Proceeding further, we find
$$\Xi_k = \hat{\delta}_k - \sum_{\ell \le k - 1} {m_\ell \over |p_k - p_\ell|}\geq  \hat{\delta}_k -  \frac{2}{|p_k|}\sum_{\ell \le k - 1} m_\ell\geq \frac12 \hat\delta_k,$$ 
$$\Xi_{k-1, k} = \Xi_{k-1} - \frac{m_{k}}{|p_{k-1}-p_k|}\geq \Xi_{k-1} - \frac{2m_k}{|p_k|} \geq \frac12 \Xi_{k-1} \geq 2^{-2} \hat \delta_{k-1},$$
while for for $q<k-1$,
$$\Xi_{q, k} = \Xi_{q, k-1} - \frac{m_{k}}{|p_k-p_{q}|}\geq \Xi_{q, k-1} -\frac{2m_{k}}{|p_k|}\geq \frac12 \Xi_{q, k-1}\geq 2^{-(k-1)} \hat\delta_q. $$

\begin{remark} For a given finite number of masses, there some $\sigma\geq 5$ such that it suffices just to pick $p_k$ so that $|p_k-p_\ell|\geq \sigma$ for $\ell\neq k$.
\end{remark}


\section{Proof of Theorem~\ref{thm:main}} \label{sec:ThmMain}

The last ingredient for the proof of the main theorem is the following proposition. 

\begin{prop} \label{prop:sm-mass}  For any $R_0>0$, and $k\in \mathbb Z_+$, there is a $\mu_0>0$ such that for each $0<m\le \mu_0$, there is a smooth metric $g^m$ on $\mathbb R^3$ with vanishing scalar curvature, for which $g^m=g^m_S$ for $|x|\geq R_0$, and with $g^m$ converging in $C^k$ on compact subsets to the Euclidean metric $\mathring g$ as $m\searrow 0$. 
\end{prop}

Using this proposition, we can prove Theorem~\ref{thm:main}.

\begin{proof}[Proof of Theorem~\ref{thm:main}]   Let $\mathfrak m$ be a sequence with $0<m_i\leq \mu_1\leq \mu_0/2$, where $\mu_0$ is as in the preceding proposition.  By the computations in the preceding section, there is a sequence $\mathfrak p$ for which for all $\lambda\geq 1$, in each annulus $\Omega_k=\{ x: 1<|x-\lambda p_k|<2\}$, the metric $g^{\mathfrak m, \lambda \mathfrak p}_{BL}$ satisfies the requirements for $g_{\mathrm{ext}}$ and $m=m_k$ and $\varepsilon>0$ from Theorem ~\ref{thm:mainlemma}.  As the construction for Theorem ~\ref{thm:mainlemma} is localized, applying it around each point gives a metric on $\mathbb R^3\setminus \lambda\mathfrak  p$ which has vanishing scalar curvature, agrees with $g_{BL}^{\mathfrak m, \lambda \mathfrak p}$ at all $x$ with $|x-\lambda p_k|\geq 2$ for all $k$, and is identical to the Schwarzschild metric $g_S^{\widehat m_k, \lambda p_k}$ with mass $\widehat m_k\leq 3\mu_0/4$ in each $\{ x: 0<|x-\lambda p_k|<1\}$; for $\mu_0\leq  3/4$, this metric contains an asymptotic end and a neighborhood of the minimal sphere (horizon), located at  $|x-\lambda p_k|= \widehat m_k/2 \leq 1/2$.  The construction is completed by applying Proposition \ref{prop:sm-mass} with $R_0=3/4$, translated to $|x-\lambda p_k|\geq 3/4$, and $\mu_1$ sufficiently small depending on $\varepsilon$, and for which the Schwarzschild interiors will contain a large neighborhood of the minimal sphere) to obtain  suitable metrics to fill in around each $\lambda p_k$.  The estimate $|\partial^I_x(g_{ij}(x) - \delta_{ij})|\leq Cm_k$ for $ |x-\lambda p_k|\leq 2$ follows from Theorem ~\ref{thm:mainlemma} and the following proof of Proposition ~\ref{prop:sm-mass}. Completeness follows readily, since any sequence $x_k$ with $|x_k|\rightarrow \infty$ is readily seen to be unbounded in the metric $g$.\end{proof}

What remains is to prove Proposition~\ref{prop:sm-mass}.  To do this, we show that the constructions in \cite{cd:as-pen, cor:as-pen} yield the existence of a family scalar-flat metrics parametrized by small $m>0$, converging to the Euclidean metric everywhere  as $m\searrow 0$, and which outside a fixed ball agree with the Schwarzschild metric $g^m_S$.  To be precise, the cited constructions give the existence of a sequence of small masses tending to zero, but with a modicum of extra care, we can show that a continuum of all sufficiently small masses are represented.  The proof below extends to all $n\geq 3$, but we focus on the $n=3$ case for simplicity of exposition.

To carry out the proof of Proposition~\ref{prop:sm-mass}, we will want to effectively control the mass of a parametrized family of scalar-flat perturbations of the Euclidean metric. This is done by first considering a path of metrics parameterized by $\sigma$ passing through the flat metric at $\sigma = 0$ and satisfying two additional properties: moving along this path does not change the scalar curvature to first order, on the one hand, while on the other hand it \emph{does} change the mass to second order. We may then make the scalar curvature vanish identically using a further conformal change, and we get estimates on the conformal factor in terms of this mass. We can then apply Proposition~\ref{prop:def0} and modulate the mass in a gluing scheme to a Schwarzschild end in such a way that we effectively track the dependence of the scalar curvature on the parameter $\sigma$, the parameter $m$ of the end, and a scale parameter $R$.  Continuity in $\sigma$ and $m$ allows us to conclude that we can achieve all positive masses in a sufficiently small neighborhood of $0$.

\begin{proof}[Proof of Proposition \ref{prop:sm-mass}]
Let $\mathring g$ be the Euclidean metric on $\mathbb R^3$ given by $\mathring g_{ij}(x)=\delta_{ij}$. To begin, with respect to the background Euclidean metric $\mathring g$, we construct a nontrivial smooth  transverse-traceless (TT) (divergence-free and trace-free) symmetric tensor $h$ with compact support. One can readily construct such a tensor; for example a very concrete proof is given in \cite{cor:as-pen} for a smooth TT tensor $\mathring h$ of compact support on Euclidean $\mathbb R^n$, and since the conditions are linear and closed under pullback by the parity operator $x\mapsto -x$, and under translation by a constant vector, for suitable $c$, we let $h$ be given in Cartesian components by $h_{ij}(x) = \mathring h_{ij}(x-c) + \mathring h_{ij}(-x-c)$, which is parity-symmetric; if $\mathring h$ were not parity-antisymmetric to start, we could take $c=0$.  Note that the linearization of the scalar curvature operator at the Euclidean metric vanishes in the direction $h$ by the TT condition: $Lh=-\Delta (\mathrm{tr}(h)) + \mathrm{div}(\mathrm{div}(h))=0$.

For each $\sigma\in \mathbb R$ with $|\sigma|$ sufficiently small, consider the metric $\gamma^\sigma$ given in Cartesian coordinates as $\gamma^\sigma_{ij}= \delta_{ij} + \sigma h_{ij}$. For any $k$, there is a $C$ such that for all such $\sigma$ and all $|I|\leq k$, $|\partial_x^I R(\gamma^\sigma)|\leq C\sigma^2$.  As such, for $|\sigma|$ small, we solve for $1+v^\sigma=u^\sigma>0$ and tending to 1 at infinity with $\Delta_{\gamma^{\sigma}} u^\sigma= \tfrac18 R(\gamma^\sigma) u^\sigma$, which means $R((u^{\sigma})^4 \gamma^\sigma)=0$.  (One can solve in appropriate weighted spaces such as in \cite{ba:mass, meyers}: the operator $\Delta_{\gamma^{\sigma}}$ is an isomorphism in such spaces, and $\Delta_{\gamma^{\sigma}}-\tfrac18 R(\gamma^\sigma)$ is a small perturbation, and hence is an isomorphism, for small $|\sigma|$, cf. Section~\ref{sec:multi-h}.) By uniqueness and parity-symmetry, $u^\sigma(x)=u^{\sigma}(-x)$.  Moreover, $u^\sigma$ is Euclidean-harmonic at infinity, and as such possesses a spherical harmonic expansion, $$u^\sigma(x)= 1+ \frac{m(\sigma)}{2|x|} + O(|x|^{-3}),$$ where coefficients of the parity-antisymmetric terms $x^i|x|^{-3}$ in the expansion must vanish.  It is elementary to see that $(u^\sigma)^4\gamma^\sigma$ is asymptotically flat, and $$16\pi m(\sigma) = - \int\limits_{\mathbb R^3} R(\gamma^\sigma)u^\sigma \; d\mu_{\gamma^{\sigma}}.$$
Since $Lh=0$, it is not hard to see $m'(0)=0$ (which also follows from the Positive Mass Theorem) and an elementary calculation \cite[Prop. 3.1]{cor:as-pen} yields the second derivative $m''(0)= \tfrac{1}{32\pi} \int\limits_{\mathbb R^3} |\nabla h|^2\; dx$. Thus there are positive constants $c_1$ and $c_2$ such that for $|\sigma|$ small, $c_1 \sigma^2 \leq m(\sigma) \leq c_2 \sigma^2$. 

Let $\rho\geq 1$ be smooth, with $\rho(x)= |x|$ for $|x|\geq 2$ (alternatively, $\rho(x)= \sqrt{1+|x|^2}$).  By the above estimate of $R(\gamma^\sigma)$ and expansion using the fundamental solution of the Laplacian (cf., e.g., \cite[Thm. 3]{meyers}),
\begin{align} \sum\limits_{|I|\leq k}  \sup\limits_{x\in \mathbb R^3} \rho(x)^{1+|I|}|\partial^I_x v^\sigma(x) |\leq C_k \sigma^2\leq C_k c_1^{-1} m(\sigma). \label{eq:wexs} \end{align}  In fact, by elementary expansion of the fundamental solution $|x-y|^{-1}$, and using that $h(y)$ is supported on some $|y|\leq K$, it follows that there is a constant $C$ such that for $R$ sufficiently large and $|x|\geq R$, and with $w^{\sigma}= v^{\sigma} - \frac{m(\sigma)}{2|x|}$,  
\begin{align}
|x|^2|w^{\sigma} (x) | + |x|^3 |\partial_x w^\sigma(x)| \leq C\sigma^2.
\label{eq:exp2} \end{align} In fact, by parity-symmetry we actually have another power of $|x|$ decay.  Thus $\widehat{g}^\sigma:=( u^{\sigma})^4\gamma^\sigma$ is uniformly close to the background Euclidean metric, on compact sets (and in fact in a weighted space globally).  Moreover it is readily seen that $\widehat{g}^\sigma$ is smooth in $\sigma$ (measured in a Sobolev or H\"older norm, weighted as appropriate if desired). By explicit calculation, 
\begin{align} 
\int\limits_{\{|x|=R\}}  \sum\limits_{i, j=1}^3( (g^{m}_S)_{ij,i}- (g^m_S)_{ii,j} )\nu^j dA =16 \pi m+ O(m^2 R^{-1}) \label{eq:bdim} \end{align}
so that using (\ref{eq:wexs})-(\ref{eq:exp2}), we find  
\begin{align} \int\limits_{\{|x|=R\}}  \sum\limits_{i, j=1}^3( (\widehat{g}^{\sigma})_{ij,i}- (\widehat{g}^{\sigma})_{ii,j} )\nu^j dA &=16 \pi m(\sigma)+ O(\sigma^2 R^{-1}). \label{eq:bdim2} \end{align}
The error estimates hold uniformly in $R$ large, and $(\sigma , m)\in J\times J'$, where $J$ is an interval of the form $[-\sigma_0, \sigma_0]$, for $0<\sigma_0<1$ sufficiently small and $J'$ is a bounded interval including $0$.

We will patch our data over a large annulus $\{x:R<|x|<2R\}$ to a Schwarzschild end with mass $m$, and then rescale to the unit annulus $\Omega=\{x:1<|x|<2\}$, with an eye toward invoking Proposition \ref{prop:def0}. We will thus have three parameters $(\sigma, m, R)$, and for $(\sigma, m)\in J\times J'$, for sufficiently large $R$, we will be able to estimate the leading order behavior of various quantities up to sufficiently small errors. (As will become clear by the end of the proof, there are alternate ways to utilize the parameter space effectively, say fixing $R$ and scaling $\sigma$ and $m$ suitably, but the approach we take seems straightforward enough.)  

To carry this out, for $R>0$, let $\phi_R:\mathbb R^3\rightarrow \mathbb R^3$ be given by $\phi_R(x)= Rx$.  Let $\psi_R (x) = \psi(x/R)$, and let $\tilde g:= \tilde g^{\sigma, m, R}$ be given by $\tilde g  = R^{-2} \phi_R^* ((1-\psi_R)\widehat{g}^{\sigma} + \psi_R g^m_S)$.  There is a constant $C>0$ depending on $k$ such that for large $R$, $(\sigma , m)\in J\times J'$, $|I|\leq k$, and all $x\in \Omega$, $|\partial^I_x( \tilde g_{ij}(x) - \delta_{ij})|\leq C R^{-1}$.  As $\tilde g$ is uniformly close to the Euclidean metric, we can apply Proposition \ref{prop:def0} to obtain $\tilde h:= \tilde h^{\sigma, m, R}$ smooth, with support on $\overline{\Omega}$, such that $R( \tilde g+ \tilde h)\in \mathrm{span}\{ \zeta, \zeta x^1, \zeta x^2, \zeta x^3\}$, where as above $\zeta$ is a rotationally symmetric bump function vanishing outside $\Omega$. Moreover, we have the estimate $\|\tilde h\|_{C^{2}}\leq C\|R(\tilde g)\|_{C^1}$ for $C$ uniform in $(\sigma, m)\in J \times J'$ and $R$ large; more precisely, for any $\alpha\in (0,1)$, $\| \tilde h\|_{C^{2,\alpha}}\leq C\|R(\tilde g)\|_{C^{0,\alpha}}$.  On the other hand, by parity symmetry the components in the linear direction must vanish: $R(\tilde g^{\sigma, m,R}+\tilde h^{\sigma , m ,R})= b(\sigma, m,R) \zeta$.  Furthermore,  $\tilde h=\tilde h^{\sigma, m,R}$ and the coefficient $b$ are continuous in $(\sigma, m,R)$.  As $\|\tilde g^{\sigma, m, R}- \tilde g^{\sigma, m', r}\|_{C^{4, \alpha}}= O (|m-m'|/R)$ is readily seen, then $\|\tilde h^{\sigma, m, R}-\tilde h^{\sigma, m', R}\|_{C^{2, \alpha}} \leq C|m-m'|/R$ follows from the Lipschitz bound in Proposition~\ref{prop:def0}.

$R(\tilde g)$ can be estimated by direct expansion.  Indeed, $\tilde g$ is conformally flat, with $\tilde g_{ij}= U \delta _{ij}$, where $U(x)=(1-  \psi(x)) (1+v^{\sigma}(Rx))^4+\psi (x) ( 1+ \tfrac{m}{2|Rx|})^4$.  As such, \begin{align} R(\tilde g)= -2U^{-2} \sum\limits_{j=1}^3 U_{,jj}+ \frac32 U^{-3} |dU|^2. \label{eq:RUex} \end{align}  Moreover  $W(x):=U (x)  - ( 1+ \tfrac{m/R}{2|x|} )^4$ can be written
\begin{align*}  (1-\psi) \Big( (v^{\sigma}(Rx)-\tfrac{m/R}{2|x|})(2+v^{\sigma}(Rx)+\tfrac{m/R}{2|x|})\big((1+v^{\sigma}(Rx))^2 + (1+\tfrac{m/R}{2|x|})^2\big)\Big),
\end{align*} 
so that using (\ref{eq:wexs}) we find a constant $C$ so that for all $(\sigma, m)\in J\times J'$ and $R\geq 1$ sufficiently large (so that in particular $mR^{-1}\leq 1$),
\begin{align} 
|W(x)|+\big|\partial W (x) \big|+ \big|\partial^2 W (x) \big| & \leq C (m+ \sigma^2)R^{-1}  \nonumber\\
|\partial^2  U | & \leq  C (m+ \sigma^2)R^{-1}  \label{eq:Uexp} \\
|\partial U|^2 & \leq C (m^2+ \sigma^4)R^{-2}. \nonumber\end{align}
We readily conclude $\|R(\tilde g)\|_{C^1} = O (m+ \sigma^2 )R^{-1}$, and thus $\|\tilde h\|_{C^{2, \alpha}} = O (m+ \sigma^2 )R^{-1}.$

We let $\bar g=\bar g^{\sigma, m, R}=\tilde g^{\sigma, m, R}+ \tilde h^{\sigma, m, R}=\tilde g+ \tilde h$, and we make the following claim: 
\begin{lemma} $F(\sigma, m, R):= \int \limits_\Omega R(\bar g)\; dx=16 \pi (m-m(\sigma))R^{-1} + \mathcal E$, where for some constant $C>0$, $\mathcal E$ is continuous in $(\sigma, m , R)$, with $|\mathcal E|\leq C (m^2+ \sigma^2)R^{-2}$.  Moreover, we have $|\mathcal E(\sigma, m, R)-\mathcal E(\sigma, m', R)|\leq C |m-m'|R^{-2}$. 
\end{lemma}
\begin{proof} The proof is a straightforward calculation using the above estimates.  We begin with the expansion of the scalar curvature $R(\bar g)= R(\tilde g)+ L_{\tilde g} \tilde h + \mathcal Q_{\tilde g}(\tilde h)$.  We include the $Q_{\tilde g}(\tilde h)$-term in $\mathcal E$ by the estimate we have on $\tilde h$.  As for $L_{\tilde g} \tilde h$, we write 
$$\int\limits_\Omega   L_{\tilde g} \tilde h \; dx=\int\limits_\Omega   L_{\tilde g} \tilde h \; (dx- d\mu_{\tilde g})$$
where we integrated by parts to get $\int\limits_\Omega  L_{\tilde g} \tilde h \; d\mu_{\tilde g}=0$.  This can also put placed into $\mathcal E$ by estimating the difference of the measures: $$|dx-d\mu_{\tilde g}|/dx = |1-U^{3/2}|\leq C(m+\sigma^2)R^{-1}.$$  Finally, we use (\ref{eq:RUex})-(\ref{eq:Uexp}) to replace $\int\limits_\Omega R(\tilde g)\; dx $ by $\int\limits_\Omega \sum\limits_{i, j=1}^3 ( \tilde g_{ij,ij}-\tilde g_{ii,jj})\; dx$, with the difference contributing another term to $\mathcal E$, and then we apply (\ref{eq:bdim})-(\ref{eq:bdim2}) to obtain the difference in the masses, up to an acceptable error term.  We have $|\mathcal E(\sigma, m, R)-\mathcal E(\sigma, m', R)|\leq C |m-m'|R^{-2}$ by the Lipschitz bound on $\tilde h^{\sigma, m, R}$.  \end{proof}  As such, $F$ is Lipschitz, and hence absolutely continuous in $m$, and so $$F(\sigma, m_2, R)-F(\sigma, m_1, R)= \int\limits_{m_1}^{m_2} \frac{\partial F}{\partial m} \; dm.$$ As $\frac{\partial F}{\partial m} >0$ a.e. for $R$ sufficiently large, we conclude $F$ is increasing in $m$. 

For each $\sigma\neq 0$, let $J'_\sigma= [m(\sigma)/2, 2m(\sigma)]\subset (0, \infty)$, and recall $c_1 \sigma^2 \leq m(\sigma)$. For any fixed $R$ sufficiently large, independent of $|\sigma|$ small, we have $F(\sigma, 2m(\sigma), R)>0$, while $F(\sigma, m(\sigma)/2, R)<0$.  By continuity,   
    $F(\sigma, M(\sigma), R)=0$ for some $M(\sigma)\in J'_\sigma$.  By the monotonicity, this $M(\sigma)$ is uniquely determined.  Moreover $\frac{\partial F}{\partial m}\geq c>0$, so that $|F(\sigma, M(\sigma'), R) |= \left|\int\limits_{M(\sigma)}^{M(\sigma')} \frac{\partial F}{\partial m}\; dm\right| \geq c|M(\sigma')-M(\sigma)|$. Letting $\sigma\rightarrow \sigma'$, we see $M(\sigma)\rightarrow M(\sigma')$, since $F$ is continuous and $F(\sigma', M(\sigma'), R)=0$. 
    
   The above scheme allows us to construct a metric which outside a ball of sufficiently large radius is identically Schwarzschild of mass $m$. By rescaling, we can accommodate any radius $R_0>0$.
    \end{proof}


\section{Vacuum initial data sets with infinitely many minimal spheres}  
\label{sec:multi-h}

In the previous section we completed the proof of Theorem ~\ref{thm:main} by smoothing out the punctured regions of Brill--Lindquist data with ``small" data, and in so doing we removed any Schwarzschild horizons in $\{x:|x-c_k|<1\}$.   In this section we prove Proposition~\ref{prop:multi-h} by modifying the above construction to obtain scalar-flat initial data on $\mathbb R^3$ with infinitely many minimal spheres.  We will employ the method of Miao \cite{miao-1h}, who constructed such data with a single minimal sphere by a method different than that of Beig-\'O Murchadha \cite{beig-om-trap}; Miao's method was extended to a finite number of minimal spheres in \cite{jc-mh}.

To prove Proposition~\ref{prop:multi-h}, we first construct metrics which are flat in some ball, which agree with an exterior end of a Schwarzschild (including a neighborhood of the minimal sphere) outside of a larger ball, and have positive (and potentially large) scalar curvature in a transition region.  Following Miao \cite{miao-1h}, we then use a result of Lohkamp \cite[Thm. 1]{loh-ham} in order to push the scalar curvature down to near $0$ in the transition region. Adjusting the scalar curvature to be identically $0$ is accomplished with a conformal change in a perturbative setting, as Lohkamp's result gives a change in metric which is small in $C^0$.  As such, the resulting metric will still admit a minimal sphere near that of the original Schwarzschild end. 

Because the scalar curvature of the resulting metric is identically $0$, we can then follow the proof of Theorem~\ref{thm:main}, as we are now in a position to apply Theorem~\ref{thm:mainlemma}. We do so in order to conclude the proof, with the slight modification that we modulate using a rescaling parameter $\theta$ instead of the Schwarzschild mass directly.

For use in the proof, we recall the definition of weighted Sobolev spaces, following \cite{ba:mass}.  As we had earlier, let $\rho\geq 1$ be a smooth function on $\mathbb R^3$ for which $\rho(x)= |x|$ for $|x|\geq 2$.  Then for $k$ a nonnegative integer, $\tau\in \mathbb R$ and $p>1$, we have weighted Sobolev spaces $W^{k,p}_{-\tau}=W^{k,p}_{-\tau}(\mathbb R^3)$ given by the norm
$$\|u\|_{W^{k,p}_{-\tau}}^p= \sum\limits_{j=0}^k \sum\limits_{|I|=j} |\partial^I u|^p \rho^{\tau p -3} \; dx.$$  

\begin{proof}[Proof of Proposition~\ref{prop:multi-h}] Let $0<m<1$, so that the minimal sphere in $(g_S^m)_{ij}(x)= (1+\tfrac{m}{2|x|})^4 \delta_{ij}$ is located at $|x|= \tfrac{m}{2}<\tfrac12$.   Fix some $0<r_0<\tfrac12$, and consider the function $$w(x)= \begin{cases} 1+\tfrac{m}{2r_0} , \qquad |x|< r_0 \\ 1+ \tfrac{m}{2|x|}, \qquad |x|\geq r_0\end{cases}.$$  Then $w$ is a Euclidean superharmonic function, and for small $\eta>0$, a mollification $w_\eta$ of $w$ by a rotationally symmetric mollifier with small enough support is then readily seen yield a smooth metric $w_\eta^4 \delta_{ij}$ with positive scalar curvature in an annulus $r_0-\eta<|x|<r_0+\eta$, and vanishing scalar curvature outside this region, agreeing with a flat metric near the origin, and with the original Schwarzschild metric in a neighborhood of $|x|>\tfrac{m}{2}$.  As in \cite{miao-1h}, we now apply a result Lohkamp \cite{loh-ham} which gives for any $\varepsilon>0$, a $C^0$-nearby smooth metric $g^{\varepsilon}$ with $-\varepsilon\leq R(g^{\varepsilon})\leq 0$ on $r_0-\eta-\varepsilon<|x|<r_0+\eta+\varepsilon$, and with $g^{\varepsilon}_{ij}= w_\eta^4\delta_{ij}$ outside this region.  

We make a slight twist on the argument of Miao at this point.  For $\theta>0$ we consider the map $\varphi_\theta(x) = \theta x$, and let $g_\theta= \theta^{-2} \varphi_\theta^* g^{\varepsilon}$.  We have $R(g_\theta)= \theta^{2} R(g^{\varepsilon})$, so that $-\theta^{2} \varepsilon \leq R(g_\theta)\leq 0$, while for $|x|>\theta^{-1} (r_0+\eta+\varepsilon)$, $(g_\theta)_{ij}(x) = (1+\tfrac{m/\theta}{2|x|})^4 \delta_{ij}$.   Thus for $\eta+\varepsilon$ small and $\theta $ near 1, $g_\theta$ agrees with $g_S^{m/\theta}$ on an open set which contains $\{ x: |x|\geq \tfrac{m/\theta}{2}\} \supset \{x:1\leq |x|\leq 2\}=\overline{\Omega}$, so that in particular it contains the minimal sphere.  

For $0<\tau<1$ and $p>1$, the operator $\Delta_{g_\theta}-\tfrac18 R(g_\theta):W^{2, p}_{-\tau} \rightarrow W^{0,p}_{-\tau-2}$ is a compact perturbation of the isomorphism $\Delta_{g_\theta}$,  and as such it has Fredholm index zero \cite{ba:mass}.    If $w\in W^{2, p}_{-\tau}$ with $(\Delta_{g_\theta}-\tfrac18 R(g_\theta) )(w)=0$, multiplying by $w$ and integrating by parts we obtain $$0=\int\limits_{\mathbb R^3} (w\Delta_{g_\theta}w-\tfrac18 R(g_\theta)w^2) \; d\mu_{g_\theta}=- \int\limits_{\mathbb R^3} (|\nabla w|_{g_\theta}^2 +\tfrac18 R(g_\theta)w^2) \; d\mu_{g_\theta}$$
where the boundary term vanishes due to the decay of $w$.  In fact since $R(g_\theta)=0$ for $|x|\geq r_0+\eta+\varepsilon$, on this region $\Delta w \in W^{0,p}_{-\tau-3}$.  As such, it admits an expansion, for some constants $C$, $A$ and $\gamma>0$  \begin{align*} \Big| w(x)-\frac{A}{|x|}\Big| + |x| \Big| \partial \Big(w(x)-\frac{A}{|x|}\Big)\Big| \leq C|x|^{-1-\gamma}\end{align*} 
as in \cite[Thm. 1.17]{ba:mass} (with an additional step to get the decay of the derivative), cf. \cite[Lemma 3.2]{sy:pmt} or \cite[Thm. 2]{meyers}.  Applying the H\"older inequality, and the Sobolev inequality $\|w\|_{L^6}^2 \leq C \|\nabla w\|^2_{L^2}$, with a uniform $C$ for $\varepsilon$ small and $\theta$ near $1$ (recall that Lohkamp's deformation is $C^0$-small), we obtain
$$\|\nabla w\|^2_{L^2} \leq\|\tfrac18 R(g_\theta)\|_{L^{3/2}} \|w\|^2_{L^6} \leq C\|\tfrac18 R(g_\theta)\|_{L^{3/2}}\|\nabla w\|^2_{L^2}\leq C'\theta^2 \varepsilon \|\nabla w\|^2_{L^2}. $$ Thus for $\theta$ near $1$ and $\varepsilon$ small, we conclude $w=0$, and hence that $\Delta_{g_\theta}-\tfrac18 R(g_\theta)$ is an isomorphism.

 We want to solve for a positive function $u=u_\theta$ close to $1$, and so that $u(x)$ approaches 1 as $|x|\rightarrow \infty$, with $R(u^4 g_\theta)=0$, i.e. $ \Delta_{g_\theta}u-\tfrac18 R(g_\theta)u=0$. We let $v=u-1$, so that the equation becomes $\Delta_{g_\theta}v-\tfrac18 R(g_\theta)v= \tfrac18 R(g_\theta)$. Since the operator $\Delta_{g_\theta}-\tfrac18 R(g_\theta)$ is an isomorphism, we can solve for $v\in W^{2,p}_{-\tau}$.   For $\theta$ near $1$ and $\varepsilon$ sufficiently small, we show that $u$ is close to $1$ in $C^2(\overline{\Omega})$, and a barrier argument as in \cite{miao-1h} guarantees there is a minimal sphere in $(\Omega, u^4 g_\theta)$.
 
We proceed to get estimates on $v=u-1$.  Since $R(g_\theta)=0$ for $|x|\geq r_0+\eta+\varepsilon$ where the barrier argument is applied, on this region $(1+v(x))(1+\frac{m/\theta}{2|x|})$ is a Euclidean harmonic function which tends to 1 as $|x|\rightarrow \infty$. As such, it admits an expansion in spherical harmonics: \begin{align*} (1+v(x))\Big(1+\frac{m/\theta}{2|x|}\Big) & = 1+ \frac{m_\theta}{2|x|} + O(|x|^{-2})\end{align*} so that $v(x) = \frac{m_\theta-m/\theta}{2|x|}+ O(|x|^{-2})$ and $|\partial v(x)|= O(|x|^{-2})$. Using these decay rates, we get integral estimates on $v$ by multiplying the differential equation by $v$ and integrating by parts as we did above,  
and applying H\"{o}lder's inequality (and the arithmetic-geometric mean inequality) to obtain for any $\delta>0$ (norms taken on $(\mathbb R^3, d\mu_{g_\theta})$),
\begin{align*}
\|\nabla v\|^2_{L^2} &\leq\|\tfrac18 R(g_\theta)\|_{L^{3/2}} \|v\|^2_{L^6}+ \|\tfrac18 R(g_\theta)\|_{L^{6/5}}\|v\|_{L^6}\\
&\leq \|\tfrac18 R(g_\theta)\|_{L^{3/2}} \|v\|^2_{L^6}+ \tfrac{1}{2\delta}\|\tfrac18 R(g_\theta)\|^2_{L^{6/5}}+ \tfrac{\delta}{2} \|v\|^2_{L^6}.
\end{align*}
Applying the Sobolev inequality $\|v\|_{L^6}^2 \leq C \|\nabla v\|^2_{L^2}$, with a uniform $C$ for $\varepsilon$ small and $\theta$ near $1$ as above, for $\delta$ sufficiently small, we can absorb the $\|v\|_{L^6}^2$ terms to obtain, where $C$ and $C'$ are uniform for $\varepsilon$ small and $\theta$ near 1:
$$\|v\|_{L^6} \leq C\| R(g_\theta)\|_{L^{6/5}}\leq C' \theta^2 \varepsilon .$$
By a DeGiorgi-Nash-Moser estimate \cite[Thm, 8.17]{gt:pde},  $$|v(x)|\leq C (\|R(g_\theta)\|_{L^3}+ \|v\|_{L^6})\leq C' \theta^2 \varepsilon.  $$
The constant $C$ in the above estimate depends on ellipticity and boundedness constants for the divergence-form operators $\Delta_{g_\theta}$, which are uniform in $\varepsilon$ small and $\theta$ near 1, by the $C^0$-control on $g_\theta$.   Since on a neighborhood of $\overline{\Omega}$ we have $(g_\theta)_{ij}=  (1+\frac{m/\theta}{2|x|})^4\delta_{ij}$, interior Schauder estimates (cf. \cite[Thm. 6.2]{gt:pde}) for $\Delta_{g_\theta}v=0$ yield $\|v\|_{C^{2,\alpha}(\overline{\Omega})} \leq C \theta^2 \varepsilon.$

It is also not hard to show that $u=u_\theta=v_\theta+1$ depends continuously on $\theta$, as 
\begin{align*}
\Delta_{g_\theta} (u_{\tau} - u_{\theta})&= \Delta_{g_\tau} u_\tau - \Delta_{g_\theta} u_\theta + (\Delta_{g_\theta} - \Delta_{g_\tau})u_\tau\\
&= \tfrac18 R(g_\tau) u_\tau - \tfrac18 R(g_\theta) u_\theta+(\Delta_{g_\theta} - \Delta_{g_\tau})u_\tau\\
&= \tfrac18 R(g_\theta) (u_\tau-u_{\theta}) + \tfrac18 \big(R(g_\tau) -R(g_\theta)\big) u_\tau+(\Delta_{g_\theta} - \Delta_{g_\tau})u_\tau\end{align*}
so that $$ \Big(\Delta_{g_\theta} -\tfrac18 R(g_\theta)\Big) (v_\tau-v_{\theta})=\tfrac18 \big(R(g_\tau) -R(g_\theta)\big) u_{\Red{\tau}}+(\Delta_{g_\theta} - \Delta_{g_\tau})u_\tau.$$
Repeating the integration-by-parts argument sketched above, allows us to conclude that $u_\theta$ varies continuously in $L^6(\mathbb R^3)$, and then by the elliptic estimates in $C^{2,\alpha}(\overline{\Omega})$.

We now apply the argument in the proof of Theorem ~\ref{thm:main}, but instead of using a  Schwarzchild as the interior metric for the Brill--Lindquist exterior, we use the metrics $u_\theta ^4 g_{\theta}$ considered here.  As such, the computations from the proof of Theorem ~\ref{thm:mainlemma} are slightly modified as follows. 

Let $m_i=m$ for all $i\in \mathbb Z_+$.  There is a sequence $\mathfrak p$ as in Section~\ref{sec:BLest} and a $0<\xi<\tfrac12$, such that for any $\theta\in J'_\xi:=[1-\xi, 1+\xi]$, the metric 
$$
\widetilde g_{ij}(x) = (1- \psi(x))(u_\theta^4(x) g_\theta)_{ij}(x) + \psi(x)(g^{\mathfrak m, \mathfrak p}_{BL})_{ij}(x).
$$
is sufficiently near $g_S^m$ to apply Proposition~\ref{prop:def} with $\gamma=\widetilde g$ to solve for $h(\tilde g)$, satisfying the conditions therein.

We write
$$
\widetilde g= (1-\psi) g_{\theta}+\psi g_S^{m}+\widehat h   =:\widehat g + \widehat h,
$$ 
so that $\widehat h := (1-\psi)(u_\theta^4-1)g_\theta+\psi (g^{\mathfrak m, \mathfrak p}_{BL} - g^m_S)$ and $\widehat g=(1-\psi) g_{\theta}+\psi g_S^{m}$, which is $(1-\psi) g_S^{m/\theta}+\psi g_S^{m}$ on the set $\overline{\Omega}=\{x:1\leq |x|\leq 2\}$.  Now, $R(\widetilde g+h(\widetilde g))$ vanishes precisely when $$\int\limits_{\Omega} f^{m} R(\widetilde g+h(\widetilde g))\; d \mu_{g^m_S}=0.
$$
We note that
\begin{align*} 
R(\widetilde g+h(\widetilde g))&= R( \widehat g)+ L_{\widehat g}( \widehat h+h(\widetilde g))+\mathcal Q_{\widehat g} (\widehat h + h(\widetilde g))\\
&= R( \widehat g )+L_{g_S^m}(\widehat h+h(\widetilde g))\\
& \qquad  +(L_{\widehat g} - L_{g_S^m})(\widehat h + h(\widetilde g))+ \mathcal Q_{\widehat g}(\widehat h+h(\widetilde g)).
\end{align*}

The integral of the term $f^m L_{g^m_S}(h(\widetilde g))$ vanishes because $f^m$ is in the kernel of $L^*_{g^m_S}$, and $h(\widetilde g)$ and its derivatives vanish along the boundary.

We can estimate as we did earlier in the proof of Theorem~\ref{thm:mainlemma}, and we emphasize the constants in the estimates below are independent of $\theta\in J'_\xi$: for some $C>0$,
\begin{align*}
    \left| \int\limits_{\Omega} f^{m} L_{g_S^{m}}(\widehat h)d\mu_{g^m_S} \right| &\leq C \|\widehat h\|_{C^2(\Omega)}\\
    |\mathcal Q_{\widehat g}(\widehat h+h(\widetilde g))| &\leq C (\|\widehat h\|^2_{C^2(\Omega)}+ \|h(\widetilde g)\|^2_{C^2(\Omega)})\\
    |( L_{\widehat g}-L_{g_S^{m}}) (\widehat h+h(\widetilde g))|&\leq  C\|g^{m/\theta}_S-g^{m}_S\|_{C^2(\Omega)} \cdot(\|\widehat h\|_{C^2(\Omega)}+ \|h(\widetilde g)\|_{C^2(\Omega)})\\
    \|h(\widetilde g)\|_{C^{2,\alpha}(\Omega)}& \leq C_0 \|R(\widetilde g)\|_{C^{0,\alpha}(\Omega)} \leq C (|m-m/\theta|+\|\widehat h\|_{C^{2,\alpha}(\Omega)})
\end{align*}

Just as we did in (\ref{eq:mint}), we can readily compute $$\int\limits_{\Omega} f^{m} R( \widehat g )\; d\mu_{g^m_S}=16\pi (m-m/\theta)= 16\pi \theta^{-1} m(\theta-1). $$

Thus we conclude that there is a constant $C>0$ such that 
\begin{align*} 
\left|\int\limits_{\Omega} f^{m} R(\widetilde g+h(\widetilde g))\; d \mu_{g^m_S}- 16\pi(m-m/\theta)\right| &\leq C ((m-m/\theta )^2+  \|\widehat h\|_{C^{2,\alpha}(\Omega)}).
\end{align*}
Given $\xi$, we can choose $\mathfrak p$ and $\varepsilon$ so that $\|\widehat h\|_{C^{2,\alpha}(\Omega)} \leq C (\theta^2 \varepsilon + \xi^2)\leq C'\xi^2$.  
By Proposition \ref{prop:def}, $$\theta\mapsto \mathcal J(\theta):= \int\limits_{\Omega} f^{m} R(\tilde g+h(\tilde g))\; d \mu_{g^m_S}
$$ 
is continuous for $\theta \in J'_\xi$.  For $\xi$ sufficiently small, we then have by the above estimates that $\mathcal{J}(1-\xi) \cdot \mathcal{J} (1+\xi) < 0$. Thus by continuity, for some $\theta\in J'_\xi$, $\mathcal J(\theta)= 0$, as desired. \end{proof}

\begin{remark} Employing the above approach in the proof of Theorem~\ref{thm:main} would only retain the $C^0$ estimate on $g_{ij}-\delta_{ij}$.
\end{remark}


\section{On the total mass} \label{sec:totmass}

  For suitably asymptotically flat exterior ends, the ADM mass is given by the following formula: 
\begin{align}  
16\pi m_{\mathrm{ADM}}(g) &= \lim\limits_{r\rightarrow \infty} \int\limits_{\{|x|=r\}} \sum\limits_{i,j=1}^3 (g_{ij,i}-g_{ii,j}) \tfrac{x^j}{|x|} dA \nonumber\\
&= \lim\limits_{r\rightarrow \infty} \int\limits_{\{|x|=r\}} g^{ab}(g_{aj,b}-g_{ab,j}) \tfrac{x^j}{|x|} dA. \label{eq:altADM} \\
&= \lim\limits_{r\rightarrow \infty} \int\limits_{\{|x|=r\}} g^{ab}(g_{aj,b}-g_{ab,j}) \nu_g\; d\sigma_g, \label{eq:altADM} \nonumber \end{align}
where $\nu_g$ the outward-pointing hypersurface normal with respect to $g$.  The ADM mass is finite when $R(g)$ is integrable on the asymptotically flat end.  More precisely, the above defines the ADM \emph{energy}, one component of the ADM energy-momentum vector, but our time-symmetric setting, the linear momentum vanishes.

We prove the following proposition on the limit of the flux integrals for the data we construct.  Given a sequence $\mathfrak c=(c_k)_{k\in \mathbb Z_+}$ in $\mathbb R^3$, we let $B_k=\{x:|x-c_k|<2\}$, and for $r>0$, we then let $B_k^r=\{ x: |x-c_k|<r/10\}$, and let $B(r)=\{ x: |x| < r\}$.  We observe that for $\lambda \geq 1$ and $\mathfrak c= \lambda \mathfrak p$, where $\mathfrak p$ is as defined recursively in Section ~\ref{sec:BLest}, and for any $r>0$, the sets $\{ k: B_k \cap \partial B(r)\neq \varnothing\}$ and $\{ k: B_k^r \cap \partial B(r)\neq \varnothing\}$ have at most one element in them, and $|c_k-c_j|\geq \tfrac12 |c_j|$ for all $j<k$.

\begin{prop}  Consider a metric $g$ constructed according to Theorem~\ref{thm:main}, with mass sequence $\mathfrak m$ with $\|\mathfrak m\|_{\ell^\infty}\leq \varepsilon_0$ and corresponding center sequence $\mathfrak c= \lambda \mathfrak p$, with $\widehat \delta_k$ as in ((\ref{eq:sum-est}) chosen so that $(\widehat \delta_k)_{k\in \mathbb Z_+} \in \ell^1$, and with a constant $C>0$ such that for $x\in B_k $ and $0\leq |I|\leq 2$, $|\partial_x^I (g_{ij}(x) -\delta_{ij})|\leq Cm_k.$

In case $\mathfrak m\notin \ell^1$ and $ \#\{ k: B_k \cap \partial B(r)\neq \varnothing\}$ is uniformly bounded for $r>0$, then the limit in (\ref{eq:altADM}) is infinite.  As such, we conclude that since $R(g)=0$, $g$ is not suitably asymptotically flat.

In case $\mathfrak m\in \ell^1$, then the limit in (\ref{eq:altADM}) is finite.  In this case we have the following:  
\begin{enumerate} \item[$\bullet$] if $\sum\limits_{k\in \mathbb Z_+} \widehat \delta_k$ is suitably small, then the limit in (\ref{eq:altADM}) is strictly positive. 

\item[$\bullet$] if $\# \{ k: B_k^r \cap \partial B(r)\neq \varnothing\}$ is uniformly bounded for  $r>0$, and if for all $j<k$, $|c_k-c_j|\geq \tfrac27 |c_j|$, then the limit in (\ref{eq:altADM}) equals $16\pi \|\mathfrak m\|_{\ell^1}.$ 
\end{enumerate} \label{prop:mass}
\end{prop}

\begin{proof} Let $\nu$ denote the outward-pointing Euclidean unit normal for a closed hypersurface; for example $\nu=\frac{x-c_k}{|x-c_k|}$ for the boundary $\partial B_k$.

The metric $g$ is exactly Brill--Lindquist outside of $\bigcup_{k\in \mathbb Z_+} B_k$, where we write $g_{ij}(x) = f(x) \delta_{ij}$, so that $f= (1+\sum_{k\in \mathbb Z_+}f_k)^4$, with $f_k(x) := \frac{m_k}{2|x-c_k|}$, and for which we readily compute (letting $\nabla$ be the Euclidean gradient)
\begin{equation}
    \sum\limits_{a,b,j=1}^3  g^{ab}( g_{aj,b}- g_{ab,j}) \frac{x^j}{|x|}  = \sum\limits_{j=1}^3- 2 \frac{f_{,j}}{f} \frac{x^j}{|x|} = -2(\nabla \log f) \cdot \frac{x}{|x|}.
\end{equation}

Let $\mathscr J(r)=\{k: B_k\cap  B(r)\neq \varnothing\}=\mathscr J_1(r) \cup \mathscr J_2(r)$, where $\mathscr J_1(r)=\{k:B_k \subset B(r)\}$ and $\mathscr J_2(r) = \{k: B_k \cap \partial B(r) \neq \varnothing\} $.
Let $\mathscr B(r)=\bigcup\limits_{k\in \mathscr J(r)} B_k$.  By the divergence theorem, then,
\begin{equation}
\begin{aligned}
    \frac{1}{16\pi} &\int\limits_{\{|x|=r\}} \sum\limits_{a,b,j=1}^3 { g}^{ab}({ g}_{aj,b}-{ g}_{ab,j}) \frac{x^j}{|x|} dA \label{eq:IBP-m-est}\\
    &= - \sum_{k\in \mathscr J_1(r)} \frac{1}{8\pi} \int_{\partial B_k} (\nabla \log f) \cdot \nu\, dA- \frac 1 {8\pi}\int_{B(r) \setminus \mathscr B(r)} \Delta \log f\; dx\\
    &\quad + \frac 1 {16\pi} \sum_{k\in \mathscr J_2(r)} \int_{B(r) \cap B_k} \sum\limits_{a,b,j=1}^3\partial_j( { g}^{ab}({ g}_{aj,b}-{ g}_{ab,j}) )  dx .  \end{aligned}
\end{equation}
Since $\Delta \log f = - \frac{|\nabla f|^2}{f^2}$, the contribution of the second integral is nonnegative.

We now calculate
\begin{align}- \frac{1}{8\pi} \int_{\partial B_k}\nabla  \log((1+ f_k)^4) \cdot \nu\, dA =  \frac{m_k}{1+m_k/4}\geq \min (\frac{m_k}{2}, 2). \label{eq:bdy-ell}
\end{align}

Recall that $f (1+f_k)^{-4} =(1+\Theta_k)^4$, where $\Theta_k=\sum_{\ell\neq k} \Theta_{\ell k}$ was defined in~\eqref{eq:kl}. By (\ref{eq:seriesbound1})-(\ref{eq:sum-est}), for $\widehat\delta_k$ small and $1\leq |x-c_k|\leq 2$, \begin{align} |\nabla \log f - \nabla \log ((1+f_k)^4)| = 4 |\nabla \log  (1+ \Theta_k(x)) |\leq C' \frac{\widehat{\delta}_k}{1-C\widehat \delta_k}\leq C''\widehat \delta_k. \label{eq:del} \end{align}

Let us first consider the case $\|\mathfrak m\|_{\ell^1}=\infty$.  Since $\mathfrak m\in \ell^\infty$, the last integral in (\ref{eq:IBP-m-est}) is uniformly bounded, since $\# \mathscr J_2(r)$ is bounded uniformly for all $r>0$.  As $\sum_\ell \widehat{\delta}_\ell <\infty$, we conclude by (\ref{eq:IBP-m-est})-(\ref{eq:del}) that  $$\lim\limits_{r\rightarrow \infty} \frac{1}{16\pi}\int\limits_{\{|x|=r\}} \sum\limits_{a,b,j=1}^3 g^{ab}(g_{aj,b}-g_{ab,j}) \frac{x^j}{|x|} dA = \infty.$$

As we have shown, it follows from condition (\ref{eq:sum-est}) that $\sum_\ell {m_\ell \over 1 + |p_\ell|}$ converges, whence the series $\sum_{k\in \mathbb Z_+} f_k$ converges, uniformly on compacts of $\mathbb R^3\setminus \mathfrak c$, and likewise for the derivative series.  As $\nabla f_k = -\frac{m_k/2}{|x-c_k|^2} \frac{ x-c_k}{|x-c_k|}$, we have 
\begin{align} |\nabla \log f|=\frac{|\nabla f|}{f}= \frac{4}{f^{1/4}} \Big|\sum\limits_{k\in \mathbb Z_+} \frac{m_k/2}{|x-c_k|^2} \frac{ x-c_k}{|x-c_k|}\Big|\leq 2 \sum\limits_{k\in \mathbb Z_+} \frac{m_k}{|x-c_k|^2} . \label{eq:gr-log} \end{align}  
For $\mathfrak m\in \ell^1$, $|\nabla \log f| \in L^2(\mathbb R^3\setminus \bigcup\limits_{k\in \mathbb Z_+} B_k)$, and so $\Delta \log f \in L^1(\mathbb R^3\setminus \bigcup\limits_{k\in \mathbb Z_+} B_k)$ (cf. Remark ~\ref{rmk:mell2}). 

We turn to the case $\mathfrak m\in \ell^1$, for which we see, recalling $|\partial_x^I (g_{ij}(x) -\delta_{ij})|\leq Cm_k$ on $B_k$, the last integral in (\ref{eq:IBP-m-est}) tends to 0 in the limit. If in addition $(\widehat \delta_k)_{k\in \mathbb Z_+}\in \ell^1$, we conclude from (\ref{eq:IBP-m-est})-(\ref{eq:del}) that $ \lim\limits_{r\rightarrow \infty} \int\limits_{\{|x|=r\}} g^{ab}(g_{aj,b}-g_{ab,j}) \tfrac{x^j}{|x|} dA$ is finite, and in fact if $\sum_{k\in \mathbb Z_+}\widehat\delta_k$ is suitably small, the limit is seen to be strictly positive.

Finally, we let $\mathscr J_3(r) = \{ k: B_k^r \cap \partial B(r)\neq \varnothing\}$.  We consider $r\geq 20$ so that $B^r_k\supset B_k$.  By assumption, if $j\neq k$ and $\min\{|c_j|, |c_k|\}\geq \tfrac{7r}{10}$, then $|c_k-c_j|\geq \tfrac{r}{5}$.  As such, if $k\in \mathscr J_3(r)$ and if $j\neq k$, then $B^r_j\cap B^r_k=\varnothing$.  Let $\Sigma_r$ be the boundary of $B(r)\setminus \bigcup_{k\in \mathscr J_3(r)} B_k^r $.    Then we have 
\begin{equation}
\begin{aligned}
    &\frac{1}{16\pi}\int\limits_{\{|x|=r\}} \sum\limits_{a,b,j=1}^3 { g}^{ab}({ g}_{aj,b}-{ g}_{ab,j}) \frac{x^j}{|x|} dA \\
    &= -  \frac{1}{8\pi} \int_{\Sigma_r} (\nabla \log f) \cdot \nu\, dA - \frac {1}{8\pi}\int_{\bigcup_{k\in \mathscr J_3(r)} (B(r)\cap (B_k^r \setminus B_k))} \Delta \log f\; dx \label{eq:IBP-m-est-2}\\
    &\quad + \frac 1 {16\pi} \sum_{k\in \mathscr J_3(r)} \int_{B(r) \cap B_k} \sum\limits_{a,b,j=1}^3\partial_j( { g}^{ab}({ g}_{aj,b}-{ g}_{ab,j}) )  dx . 
\end{aligned}
\end{equation}
Since $\Delta \log f = -\frac{|\nabla f|^2}{f^2} \in L^1 (\mathbb R^3 \setminus \bigcup_{k\in \mathbb Z_+} B_k)$ as shown above, we see the second integral tends to 0 in the limit. As above, the final integral above tends to 0 in the limit, since $\mathfrak m\in \ell^1$ and $|\partial_x^I (g_{ij}(x) -\delta_{ij})|\leq Cm_k$ on $B_k$.  

As for the first term, we have $ \nabla \log f = \frac{4}{f^{1/4}} \sum\limits_{k\in \mathbb Z_+} \nabla f_k$.  For any $\delta>0$, there is an $N$ such that $\sum\limits_{\ell \geq N+1} m_\ell<\delta$, and then for $r\geq R$ sufficiently large, $\sum_{\ell\leq N} f_\ell < \delta$ on $\Sigma_r$.  Thus given $\varepsilon>0$, there is an $R$ such that for all $r\geq R$, $|f^{-1/4}-1|\leq \varepsilon$ on $\Sigma_r$. Since $\varepsilon$ is arbitrary, it follows that
$$\lim\limits_{r\rightarrow \infty}-\frac{1}{8\pi} \int_{\Sigma_r} \frac{4}{f^{1/4}}\;  \nabla f_k \cdot \nu \; dA = m_k.$$ 
On the other hand, by the bound on $\#\mathscr J_3(r)$, there is a constant $C>0$ so that the area of $\Sigma_r$ is at most $Cr^2$.  As such, $\int_{\Sigma_r} |x-c_k|^{-2} dA$ is uniformly bounded for $k\in \mathbb Z_+$ and $r\geq 20$, and using $\mathfrak m\in \ell^1$ together with (\ref{eq:gr-log}, we can conclude the limit in (\ref{eq:IBP-m-est-2}) is $\|\mathfrak m\|_{\ell^1}$ as desired. 
\end{proof}

\begin{remark} \label{rmk:mell2}  It is interesting to note that having $\mathfrak{m} \in \ell^2$ is the most natural condition in order to ensure $\Delta \log f=-f^{-2} |\nabla f|^2\in L^1(R^3\setminus \bigcup\limits_{k\in \mathbb Z_+} B_k)$ in our setting, as the localization makes the contributions of the $|\nabla f_k|$ almost $L^2$-orthogonal. Indeed, setting $\mathcal{R} = \mathbb R^3\setminus \bigcup\limits_{k\in \mathbb Z_+} B_k$, we have that each function $|\nabla f_k|={m_k \over |x - c_k|^2}$ is in $L^2 (\mathcal{R})$. Moreover, we have that (where $A\lesssim B$ when $A\leq CB$ for some constant $C$) 
\[
\int_{\{ |x - c_k| \ge 2\} \cap \{ |x - c_j| \ge 2\}} {1 \over |x - c_k|^2 |x - c_j|^2} d x \lesssim {1 \over 1 + |c_k - c_j|} \lesssim 2^{-|j - k|}.
\]
This means that for $j\neq k$, $|\nabla f_k|$ and $|\nabla f_j|$ are almost $L^2$-orthogonal. Thus, we have that 
\begin{equation*}
\begin{aligned}
\Vert \nabla \log f \Vert_{L^2 (\mathcal{R})}^2 &\lesssim \int_\mathcal{R} \sum_{j\geq 1} \sum_{k\geq 1} |\nabla f_j| |\nabla f_k| d x \lesssim \sum_{j\geq 1} \sum_{k\geq 1} m_j m_k 2^{-|j - k|} \le 3 \Vert \mathfrak{m} \Vert_{\ell^2}^2 
\end{aligned}
\end{equation*}
since 
\begin{align*}
\sum_{j\geq 1} \sum_{k\geq 1} m_j m_k 2^{-|j - k|}& = \|\mathfrak m\|_{\ell^2} + \sum_{k\geq 1} \sum_{j> k}2 m_j m_k 2^{-j + k}\\
&\leq \|\mathfrak m\|_{\ell^2}+ \sum_{j= 2}^\infty m_j^2 \sum_{k=1}^{ j-1} 2^{-j+k} +  \sum_{k\geq 1} m_k^2\sum_{j> k}  2^{-j + k} \leq 3 \|\mathfrak m\|_{\ell^2}.
\end{align*}

We remark that the above estimate uses that the points $c_k$ move away from each other geometrically. If the points move away more slowly, the orthogonality will be weaker, and the argument may require that $\mathfrak{m} \in \ell^p$ for some $p < 2$ in order to compensate for this.
\end{remark}

\subsection{On the Hamiltonian} Since the data we construct may fail to be asymptotically flat, one might question what the mass integrals even mean in this context.   For an interpretation, we consider the evolution of our initial data into a spacetime, expressed in lapse-shift form on a spacetime open set $\mathscr S$ as
$$\bar g = - N^2 dt^2 + g_{ij}(dx^i+X^i dt)\otimes(dx^j+ X^j dt). $$ 
The Einstein-Hilbert action is then $\int_{\mathscr S} R(\bar g)\; d\mu_{\bar g}$.  Recasting this in Hamilonian form leads one, after what is by now a classical computation, to the Hamiltonian (up to a multiple of by a constant to get dimensions of energy, which we take to be $\tfrac{1}{16\pi}$) $$\int\limits_{\mathbb R^3} \left( -N (R(g) - |K|^2_g+ (\mathrm{tr}_g K)^2) +2X \cdot \mathrm{div}_g (K-(\mathrm{tr}_gK) \right)\; d\mu_g$$ where our convention is that as a linear operator on each $T_p\mathbb R^3$, $K$ corresponds to the shape operator $-\nabla_{\bar g} n$, with $n$ the future-pointing unit normal.  As such, a simple computation shows the evolution for $g$ is $\frac{\partial g_{ij}}{\partial t} = -2NK_{ij} + X_{i;j}+ X_{j;i}$. For this to yield the same form of the equations of motion in both the setting of a closed initial slice and the setting here, we have to analyze the contribution of certain boundary terms in the analysis of the linearization of the Hamiltonian, namely those coming from the linearization $L_g$ of the scalar curvature operator, and from the divergence of the variation of the momentum tensor $K-\mathrm{tr}_g K$.  For simplicity we work in Fermi coordinates so that $N=1$ and $X=0$, in which case the term to analyze for boundary contributions (``at infinity'' here) is just $\int\limits_{\mathbb R^3} - L_g h d\mu_g $, where we consider $h$ such that $h$ and $\partial h$ tending to zero as $r$ grows without bound.  Should there be any contribution from infinity, we can just modify the Hamiltonian by subtracting off the appropriate term.  As such, the term to add comes from an analysis of $\int_{\mathbb R^3}  L_g h \; d\mu_g$.  

Let $g$ be one of the metrics constructed in Theorem~\ref{thm:main}, with $|\partial_x^I(g_{ij}(x)-\delta_{ij}|\leq C m_k$ on each $B_k$, and with $\# \mathscr{J}_2(r)$ uniformly bounded for $r>0$, where again $\mathscr{J}_2(r)=\{ k: B_k\cap \partial B(r)\neq \varnothing\}$, and $\mathscr{B}_2(r)=\bigcup\limits_{k\in \mathscr{J}_2(r)} B_k$.  As $L_gh = -\Delta_g (\mathrm{tr}_g h) + \mathrm{div}_g \mathrm{div}_g h - h \cdot \mathrm{Ric}(g)$, upon integrating over a large coordinate ball, we get the boundary term
$$\int\limits_{\{|x|=r\}} (-g^{ik} h_{ik;j} + g^{ik} h_{ij;k}) \cdot\nu^j_g d\sigma_g.$$
Since $h$ and $\partial h$ decay to 0, then since $\mathfrak m\in \ell^\infty$ and $\#\mathscr{J}_2(r)$ is bounded, we see
\begin{align} \lim\limits_{r\rightarrow\infty}\int\limits_{\{|x|=r\}\cap \mathscr{B}_2(r)} (-g^{ik} h_{ik; j} + g^{ik} h_{ij;k}) \cdot\nu^j_g d\sigma_g=0. \label{eq:hb0} \end{align}

Now on $\Sigma'_r:=\{ |x|=r\}\setminus \mathscr{B}_2(r)$, the metric is given by $g_{ij}= f\delta_{ij}$ as above, and so $d\sigma_g= f \; dA$ and the unit normal is $\nu_g = \nu/\sqrt{f}$, so that the integral becomes
\begin{equation} \label{eq:MassInt}
\begin{aligned}
\int\limits_{\Sigma'_r}\sum_{i,j} & \Big( -  f^{-1}  (h_{ii,j}- \sum_k (\Gamma^k_{ji}h_{ki}+ \Gamma^k_{ij}h_{ik})\Big)  \frac{x^j}{|x|}\, \sqrt{f}\, dA\\
& \quad + \int\limits_{\Sigma'_r}\sum_{i,j}  \Big(   f^{-1}(h_{ij,i}- \sum_k (\Gamma^k_{ii}h_{kj}+ \Gamma^k_{ij}h_{ik})\Big)  \frac{x^j}{|x|}\, \sqrt{f}\, dA\\
& = \int\limits_{\Sigma_r'}\sum_{i,j}   \Big(h_{ij,i}-h_{ii,j}\Big) f^{-1/2} \frac{x^j}{|x|}\, dA\\
&\quad +\int\limits_{\Sigma_r'}\sum_{i,j,k} \Big(-\Gamma^k_{ii}h_{kj}+ \Gamma^k_{ij}h_{ik}\Big) f^{-1/2} \frac{x^j}{|x|}\, dA
\end{aligned}
\end{equation}
where $\Gamma^k_{ij} = \frac12 f^{-1} ( \delta_{ik} f_{,j}+ \delta_{jk}f_{, i}- \delta_{ij} f_{,k})$, so that $\Gamma^k_{ii} = \frac12 f^{-1} ( 2 \delta_{ik} f_{,i}- f_{,k})$.

In order to estimate the second integral, we estimate $\int_{\Sigma'_r} f^{-1} |\nabla f|\; dA$.  With (\ref{eq:gr-log}) in mind, we claim for $\mathfrak m\in \ell^1$, and with suitable decay on $h$, the second integral above limits to 0.  Indeed, for $k\in \mathscr J_2(r)$, let $\gamma = |c_k|/r$, so $1-2r^{-1}< \gamma< 1+2r^{-1}$. Then, we note that
\[
|x - c_k|^2 = |x|^2 + |c_k|^2 - 2 |x| |c_k| \cos\phi,
\]
where $\phi$ is the angle between $x$ and $c_k$. Working in spherical coordinates, one then readily computes the integral
$$\int\limits_{\{ |x|=r\} \setminus B_k} |x-c_k|^{-2}\; dA= 2\pi \int\limits_{\phi_1}^{\pi} \frac{\sin \phi}{ 1+ \gamma^2 - 2 \gamma \cos \phi}\; d\phi = 2\pi \gamma^{-1} \log((1+\gamma) r/2),$$
where $\cos\phi_1 = \frac{1+\gamma^2-4r^{-2}}{2\gamma}$.  If $k\notin \mathscr J_2(r)$, then with $\gamma= |c_k|/r$ (so that $\gamma> 1+2r^{-1}$ or $\gamma< 1-2r^{-1}$), and assuming $c_k\neq 0$ (else the integral below evaluates to $4\pi$), we have \begin{align*}
\int\limits_{\{ |x|=r\} } |x-c_k|^{-2}\; dA& = 2\pi \int\limits_{0}^{\pi} \frac{\sin \phi}{ 1+ \gamma^2 - 2 \gamma \cos \phi}\; d\phi\\ &
= \pi \gamma^{-1}  \log  (1+\gamma^2 - 2\gamma \cos \phi)\Big|_0^{\pi}\\
&=2\pi \gamma^{-1} \log \left( \frac{ 1+\gamma}{|1-\gamma|}\right) \leq 4\pi \log (r+1),
\end{align*}
for $r\geq 4$.  

For example, then, using (\ref{eq:gr-log}), the bound on $\# \mathscr J_2(r)$ and $\mathfrak m\in \ell^1$, and if we assume $\lim\limits_{|x|\rightarrow \infty}|h_{ij}(x)|\log |x|=0$, for all $i$ and $j$, we conclude as desired.

We can recast the condition $\lim_{|x| \rightarrow \infty} |h_{i j}(x)| \log |x| = 0$ in an equivalent and manifestly translation-invariant form $\sup_{p \in \mathbb{R}^3} \lim_{|x| \rightarrow \infty} |h_{i j} (x)| \log |x-p|=0$. Another translation-invariant condition that one might naturally impose would be $h \in W^{1,p} (\mathbb{R}^3)$, for some $1\leq p < \infty$. 
Indeed, this condition allows us to show that the final integral in \eqref{eq:MassInt} tends to $0$ as $|x| \rightarrow \infty$.  For any $h \in W^{1,p} (\mathbb{R}^3)$, given $\varepsilon>0$, for sufficiently large $R$ 
\[
\int_{\{  |x| \ge R\}} (|h_{i j} (x)|^p + |\partial h_{i j} (x)|^p )d x \le \varepsilon. 
\]
Assuming for the moment that $h$ is smooth, then for any $r_2 > r_1 \geq 1$, with $\omega\in \mathbb S^2$, the unit sphere with round unit metric $\mathring g$, we have (applying H\"older's inequality)

\begin{align*}
\Big| \int\limits_{\{|x|=r_2\}} & |h_{i j} (x)|^p d A  - \int\limits_{\{|x|=r_1\}} |h_{i j} (x)|^p d A \Big| = \Big|  \int_{\mathbb S^2} \int_{r_1}^{r_2} \frac{\partial}{\partial r} (|h_{i j}(r\omega)|^{p} r^2 )dr\;  d \sigma_{\mathring g} \Big| \\
&\leq  \int\limits_{\{ r_1\leq |x|\leq r_2\}}p |h_{i j} (x)|^{p - 1} |\partial_r h_{i j} (x)| d x +  \int\limits_{\{ r_1\leq |x|\leq r_2\}} 2|x|^{-1} |h_{i j} (x)|^{p }  d x\\
&\leq  p \Big(\int\limits_{\{ r_1\leq |x|\leq r_2\}} |h_{i j} (x)|^{p } dx\Big)^{(p-1)/p} \Big( \int\limits_{\{ r_1\leq |x|\leq r_2\}} |\partial_r h_{i j} (x)|^p dx\Big)^{1/p}\\ 
& \qquad + 2\int\limits_{\{ r_1\leq |x|\leq r_2\}}  |h_{i j} (x)|^{p }  d x.
\end{align*} 
By density of smooth functions in $W^{1,p}(\mathbb R^3)$, the above estimate holds for almost every $r\geq 1$, and actually applying the Sobolev trace theorem, it can be construed to hold for all $r$.  Thus we conclude $\lim\limits_{r\rightarrow \infty} \int_{\{|x|=r\}} |h_{i j} (x)|^p d A$ exists, from which we further conclude this limit must vanish.

Now, given some $k\in \mathscr J_2(r)$, and with $\gamma$ defined as above, we can once again work in spherical coordinates, giving us that (where again $\cos\phi_1 = \frac{1 + \gamma^2 - 4 r^{-2}}{2\gamma}$)
$$\int\limits_{\{ |x|=r\} \setminus B_k} |x-c_k|^{-2 q}\; dA= 2\pi \int\limits_{\phi_1}^{\pi} \frac{\sin \phi}{ r^{2 (q - 1)} (1+ \gamma^2 - 2 \gamma \cos \phi)^q}\; d\phi \leq C_q$$
for $q > 1$; we also note that $\||x-c_k|^{-2}\|_{L^\infty(\Sigma_r')}\leq 1/4$.  By taking $q$ such that ${1 \over p} + {1 \over q} = 1$ and using Holder's inequality, the last integral in (\ref{eq:MassInt}) is bounded above by 
\[
C_p' \Vert h \Vert_{L^p (\{|x|=r\})} \rightarrow 0
\]
where $C_p'$ is some constant depending on $p$, as desired.  We note that replacing the pointwise decay assumptions on $h$ and $\partial h$ with $h\in W^{2,p}(\mathbb R^n)$, for some $1\leq p < \infty$, we could similarly handle the limit in (\ref{eq:hb0}).

The same analysis shows that the linearization of  $$\lim\limits_{r\rightarrow \infty}\int\limits_{\{|x|=r\}} g^{ab} \Big(g_{aj,b}-g_{ab,j}\Big)  \nu_g d\sigma_g$$ about a metric of the type we have constructed gives the desired limiting term. 

Using $\mathfrak m\in \ell^1$, we see  $f$ on $\Sigma_r'$ tends to $1$ as $r$ grows without bound.  Thus the limiting term we must add is thus seen to be 
$$\frac{1}{16\pi} \lim\limits_{r\rightarrow \infty}\int\limits_{\{|x|=r\}} g^{ab} \Big(g_{aj,b}-g_{ab,j}\Big)  \frac{x^j}{|x|}\, dA =\frac{1}{16\pi} \lim\limits_{r\rightarrow \infty}\int\limits_{\{|x|=r\}} \Big(g_{ij,i}-g_{ii,j}\Big)  \frac{x^j}{|x|}\, dA.$$
As such, under the conditions in Proposition ~\ref{prop:mass}, the mass of the metric can be identified with $\|\mathfrak m\|_{\ell^1}$, as shown above.

The total mass (energy) of a spacetime is typically associated to isolated systems and computed on a spacelike hypersurface which is suitably asymptotically flat (or hyperboloidal).  We have argued that for the metrics constructed in Theorem~\ref{thm:main}, which tend to fail to be asymptotically flat, one can interpret the ``sum of masses" $\|\mathfrak m\|_{\ell^1}$ as a suitable notion of mass.  While the computations presented in this section are tailored to the specific metrics we have constructed in Theorem~\ref{thm:main}, they do suggest that there may be an interesting class of metrics extending beyond the asymptotically flat regime which still admits a useful notion of mass-energy.

\appendix

\addtocontents{toc}{\setcounter{tocdepth}{1}}

\section{On the proof of Proposition \ref{prop:def}} \label{sec:proof-prop}

Recall  $\Omega=\{ x\in \mathbb R^3: 1< |x|< 2\}$ and the definition of the rotationally symmetric bump function $\zeta$ above, and we let $\mathring g$ be the Euclidean metric with $\mathring g_{ij}=\delta_{ij}$ in the coordinates $x$. 

We will prove the following extension of Proposition ~\ref{prop:def}.  To facilitate the statement, we extend some notation: let $k_0=4$, and for $m>0$, let $k_m=1$.  We now consider $g_S^m$ for $m\geq 0$, where we note $g^0_S=\mathring g$ is the Euclidean metric.  For $m>0$, $f^m(x)=\frac{ 1-\frac{m}{2|x|}}{1+\frac{m}{2|x|}}$ spans the one-dimensional kernel of $L^*_{g^m_S}$, whereas $\mathrm{ker}(L^*_{\mathring g})=\mathrm{span}\{ 1, x^2, x^2, x^3\}$.  So we let $f^0=\langle 1, x^1, x^2, x^3\rangle$, in order that we can express a general element  of $\mathrm{ker}(L^*_{\mathring g})$ as $b  f^m$, with $b\in \mathbb R^{k_m}$: in case $m=0$, $b f^0$ is the dot product.  

\begin{remark} While we only require $m\geq 0$ for our applications, for $m<0$, note that $(g_S^m)_{ij}(x)= (1+\frac{m}{2|x|})^4 \delta_{ij}$ is a scalar-flat metric on the region $|x|> -m/2$, and the kernel of $L^*_{g_S^m}=\mathrm{span}\{ f^m\} $ has dimension $k_m=1$.  The results in this section are formulated on $\Omega$, and apply equally well to the case $m>-2$.
\end{remark}

\begin{prop} Let $\Omega'$ be open and compactly contained in $\Omega$.   There is a constant $C_0>0$ such that for smooth metrics $\gamma$ sufficiently near $g_S^{m}$ in $C^{4,\alpha}(\overline{\Omega})$ and with scalar curvature $R(\gamma)$ supported in $\overline{\Omega'}$, there exists a constant $ b(\gamma)\in \mathbb R^{k_m}$ along with a smooth symmetric tensor $h=h(\gamma)$ (which extends smoothly by $0$ outside $\Omega$), so that $\gamma+h(\gamma)$ is a metric with $R(\gamma+h(\gamma))= b(\gamma)  \zeta f^{m}$.   Moreover, $h(\gamma)$, and $b(\gamma)$ depend continuously on $\gamma$, with estimates $\|h(\gamma)\|_{C^{2,\alpha}} \leq C_0 \|R(\gamma)\|_{C^{0, \alpha}}$, and $\|(h(\gamma_1), b(\gamma_1))-(h(\gamma_2),b(\gamma_2))\|_{C^{2,\alpha}\times \mathbb R^{k_m}} \leq C_0 \|\gamma_1-\gamma_2\|_{C^{4, \alpha}}$. Furthermore for each $k\in \mathbb Z_+$ there is a $C_k>0$ such that $\|h(\gamma)\|_{C^{k+2,\alpha}} \leq C_k \|R(\gamma)\|_{C^{k, \alpha}}$.

In case $\gamma$ is parity-symmetric, $h(\gamma)$ is parity-symmetric as well. 

\label{prop:def0} \end{prop}  

We briefly recall the strategy for proving this result, which is then carried out in the rest of the Appendix. Our goal is to solve for a metric with vanishing scalar curvature modulo the kernel of the adjoint of the linearized scalar curvature operator around a given metric. This will be done by iteration, repeatedly solving the equation linearized around a fixed metric (see Section~\ref{sec:ANonlinear}). In order to do this, we must effectively solve the linearized equation. This is carried out in Section~\ref{sec:ALinearized} using Hilbert space methods. However, we want to solve the linearized equation in such a way that this procedure does not change the metric in the exterior of $\Omega$.  Having nontrivial compact solutions of this kind is not possible for elliptic equations in general, but in this case, we can construct such solutions by leveraging the underdetermined nature of the equations. From a technical viewpoint, there is an elliptic estimate which does not contain boundary terms (see Section~\ref{sec:AEllipticEst}), and this estimate allows us to work in spaces which have weights that vanish, dual to spaces with weights that blow up, at all orders as we approach the boundary of $\Omega$ (see Sections~\ref{sec:AWeightedSpaces} through~\ref{sec:AWeightedEsts2}). Working in these spaces naturally localizes the gluing to $\Omega$, as desired.

\subsection{Basic injectivity estimates} \label{sec:AEllipticEst} We prove two basic injectivity estimates for $L_{g}^*$ for metrics $g$ near $g_S^m$, on suitable spaces transverse to $\mathrm{ker}(L_{g_S^m}^*)$ on $\Omega$.  For $\varepsilon\in [0, 1)$, we let $\Omega_\varepsilon=  \{ x\in \mathbb R^3: 1+\varepsilon< |x|< 2-\varepsilon\}$.

\begin{lemma}  Suppose $S\subset L^2(\Omega,g_S^m)$ is a subspace so that $L^2(\Omega,g_S^m)= S\oplus \mathrm{ker}(L_{g_S^m}^*)$.  There is a $C^{2}(\overline{\Omega})$-neighborhood $\mathcal U$ of $g_S^m$ and a $C>0$ so that for all $g\in \mathcal U$ and for all $u\in H^2(\Omega,g) \cap S$, 
\begin{equation}\label{eq:b-est-0} \|u\|_{H^2(\Omega,g)}\leq C \|L_{g}^*u\|_{L^2(\Omega,g)}.  \end{equation}
\label{lem:eq:b-est-0}
\end{lemma}

\begin{proof} Since $\mathrm{tr}_g(L_g^*u)= -(n-1) \Delta_g u - uR(g)$, we see immediately that there is a $C>0$ and a $C^{2}(\overline{\Omega})$-neighborhood $\mathcal U_0$ of $g^m_S$ so that for all $g\in \mathcal U_0$, for all $u\in H^2(\Omega,g) $, and for all $\varepsilon\in [0,1)$, 
\begin{equation*}  \|\mathrm{Hess}_g u\|_{L^2(\Omega_\varepsilon,g)}  \leq C \big(\|L_g^*u\|_{L^2(\Omega_\varepsilon,g)} + \|u\|_{L^2(\Omega_\varepsilon,g)}\big).
\end{equation*}  
We can choose such a $C\geq 1$, and hence 
\begin{equation} \label{eq:B-Easy} \|u\|_{H^2(\Omega_\varepsilon,g)}\leq C (\|L_g^*u\|_{L^2(\Omega_\varepsilon,g)} + \|u\|_{H^1(\Omega_\varepsilon,g)}). \end{equation}

If (\ref{eq:b-est-0}) fails, there is a sequence $g_i\in \mathcal U_0$, $g_i\rightarrow g^m_S$ in $C^2(\overline{\Omega})$, and a sequence $u_i\in H^2(\Omega, g_i) \cap S$ such that  $\|u_i\|_{H^2(\Omega,g_i)}=1$ but $\|L_{g_i}^*u_i\|_{L^2(\Omega,g_i)}\rightarrow 0$.  Thus there are $0<\beta_1<1<\beta_2$ such that $\beta_1\leq \|u_i\|_{H^2(\Omega, g^m_S)} \leq \beta_2$; moreover, $\|L^*_{g^m_S}u_i\|_{L^2(\Omega, g^m_S)}\rightarrow 0$.  By the Rellich Lemma and (\ref{eq:B-Easy}) on $\Omega$ with $g=g^m_S$, applied to differences $u_i-u_j$, we have that $u_i$ converges to some $u\in H^2(\Omega, g^m_S)$, with $\|u\|_{H^2(\Omega, g^m_S)} \geq \beta$, and $L_{g^m_S}^*u=0$.  Moreover, $u\in S$ as well, since $S$ is closed.  Thus we have a contradiction. \end{proof}

\begin{remark}  If $\widetilde{m}$ is sufficiently close to $m$, then $g_S^{\widetilde{m}}\in \mathcal U$.  We remark that for $u\in L^2(\Omega,g)$, $u\in H^2(\Omega,g)$ if and only if $L_g^*u\in L^2(\Omega,g)$. 
Furthermore, $f\in H^2(\Omega, g)$ if and only if $f\in H^2(\Omega, g^m_S)$, and moreover one gets an equivalent estimate if one changes the metric used for the $H^2$ and $L^2$ norms. \end{remark}

\begin{remark} For $c\in \mathbb R^3$, let $\varphi_c(x)= x-c$, and $c+\Omega= \{ x: 1< |x-c|<2\}$, so that $\varphi_c:(c+\Omega)\rightarrow \Omega$.  Let $g^{m, c}_S= \varphi_c^*(g^{m}_S)$, i.e. $(g^{m,c}_S)_{ij}(x)= (1+\frac{m}{2|x-c|})^4 \delta_{ij}$.  Then for $m>0$, the kernel of $L^*_{g^{m,c}_S}$ is spanned by $f^{m,c}$, where $f^{m,c}(x)= f^m\circ \varphi_c(x)= f^m(x-c)$, and similarly for $m=0$.   Furthermore, let $S_c$ 
and $\mathcal U_c$ be the pullbacks of $S$ 
and $\mathcal U$ under $\varphi_c$; note $g^{m,c}_S\in \mathcal U_c$ and $L^2(c+\Omega, g^{m, c}_S)= S_c \oplus \mathrm{span} (f^{m,c})$.  Then (\ref{eq:b-est-0}) 
holds for all $g\in \mathcal U_c$ and $u\in H^2(c+\Omega, g)\cap S_c$.

Given $\mathfrak m=(m_k)_{k\in \mathbb Z_+}$ with $m_k>0$ and suitably chosen $\mathfrak c$, if we now let $\widetilde{\mathcal U}^k$ be the corresponding neighborhoods for $m=\widetilde m_k$ on $\Omega$, the calculations in Section ~\ref{sec:BLest} show that if for each $k$, $\widetilde m_k$ is close enough to $m_k$, $g_{BL}^{\mathfrak{m}, \mathfrak c}\in \bigcap\limits_{k\in \mathbb Z_+} \widetilde{\mathcal U}^k_{c_k}$.  For the case of finitely many points, it suffices that $\mathfrak {\widetilde m}$ is close to $\mathfrak m$ and $|c_k-c_\ell|\geq 5$ for $k\neq \ell$.  \end{remark}

\begin{lemma} Let $\gamma$ be a metric in $C^0(\overline{\Omega})$. Let $S_\gamma$ be the $L^2(\Omega,\gamma)$-orthogonal complement of $\big( \zeta\, \mathrm{ker}(L_{g_S^m}^*)\big)$. There is a $C^{2}(\overline{\Omega})$-neighborhood $\mathcal U$ of $g_S^m$ and a $C>0$ so that for all $\varepsilon\in [0,\tfrac18]$, for all $g\in \mathcal U$ and for all $u\in H^2(\Omega,g) \cap S_\gamma$, we have
\begin{equation}\label{eq:b-est-0-eps} \|u\|_{H^2(\Omega_\varepsilon,g)}\leq C \|L_{g}^*u\|_{L^2(\Omega_\varepsilon,g)} .  \end{equation}\label{lem:b-est-0-eps}
\end{lemma} 

\begin{proof} We first recall (cf. \cite[Ch. 7]{gt:pde}) that there is a constant $D>0$ and for each $\varepsilon\in [0,\tfrac18]$ an extension operator $E_{\varepsilon}: H^2(\Omega_{\varepsilon},\mathring g)\rightarrow H^2(\Omega,\mathring g)$ such that for all $u \in H^2(\Omega_{\varepsilon},\mathring g )$, \begin{align} \|E_\varepsilon(u)\|_{H^2(\Omega,\mathring g)}\leq D\|u\|_{H^2(\Omega_{\varepsilon},\mathring g)} .\label{eq:ext-op} \end{align}  
We conclude that we can choose $D>0$ and a neighborhood $\mathcal U_0$ of $g_S^m$ so that the extension operator satisfies the estimate (\ref{eq:ext-op}) with $\mathring g$ replaced by $g\in \mathcal U_0$. 

The preceding lemma handles the case $\varepsilon=0$,  So if the claim fails, there is a sequence $g_i\rightarrow g_S^m$ in $C^2(\overline{\Omega})$ and $\varepsilon_i \in (0,\tfrac18]$ as well as $u_i\in H^2(\Omega,g_i)\cap S_\gamma$ such that $$ \|u_{i}\|_{H^2(\Omega_{\varepsilon_i},g_i)}>i \|L_{g_i}^*u_{i}\|_{L^2(\Omega_{\varepsilon_i},g_i)}.$$  
 Let $u_{\varepsilon_i}=E_{\varepsilon_i}( (u_i){\upharpoonright_{\Omega_{\varepsilon_i}}})$.   We can rescale to arrange $1= \|u_{\varepsilon_i}\|_{H^2(\Omega,g_i)}$, and so it follows that   $\|L_{g_i}^*u_{i}\|_{L^2(\Omega_{\varepsilon_i},g_i)}\rightarrow 0$.  There are constants $0<\beta_1<\beta_2$ with $\beta_1\leq \|u_{\varepsilon_i}\|_{H^2(\Omega,g_S^m)}\leq \beta_2$. In particular, then,  $\|u_i\|_{H^2(\Omega_{\varepsilon_i},g_S^m)}\leq \beta_2$, and it follows that  $\|L_{g_S^m}^*u_{i}\|_{L^2(\Omega_{\varepsilon_i},g_S^m)}\rightarrow 0$.

On the other hand by the estimate (\ref{eq:ext-op}) of the extension operator along with (\ref{eq:B-Easy}), we have 
 $$\|u_{\varepsilon_i}\|_{H^2(\Omega,g_S^m)}\leq D \| u_i\|_{H^2(\Omega_{\varepsilon_i},g_S^m)}\leq DC (\|L_{g_S^m}^*u_i\|_{L^2(\Omega_{\varepsilon_i},g_S^m)} + \|u_i\|_{H^1(\Omega_{\varepsilon_i},g_S^m)}).$$
By applying Banach-Alaoglu, and Rellich's Lemma, we conclude $u_{\varepsilon_i}$ converges $H^2(\Omega,g_S^m)$-weakly to some $u\in H^2(\Omega,g_S^m)$, and strongly to $u$ in $H^1(\Omega,g_S^m)$.  By the preceding estimate, $u_{\varepsilon_i}$ converges to $u$ strongly in $H^2(\Omega,g_S^m)$, and so we conclude $\|u\|_{H^2(\Omega, g_S^m)}\geq \beta>0$.  Furthermore, $L_{g_S^m}^* u=0$, so $u\in \mathrm{span}\{f^m\}$.   

Now, by assumption, for all $i$, and for all $f\in \mathrm{ker}(L_{g_S^m}^*)$,  $\int\limits_\Omega u_i \zeta f d\mu_{\gamma}=0$.  
Thus since $\zeta$ vanishes outside $\Omega_{\varepsilon_i}$, we conclude $u_{\varepsilon_i}\in S_\gamma$.  Since $S_\gamma$ is closed in $L^2(\Omega, \gamma)$, we have by the $L^2$-convergence that $u\in S_\gamma$.  Since we also have $u\in \mathrm{ker}(L_{g_S^m}^*)$, we conclude $u=0$, which is a contradiction.  \end{proof}

\subsection{Weighted Function Spaces} \label{sec:AWeightedSpaces}  We will use several weighted function spaces in the analysis, which we recall here.

We first define a weight function $\rho$ on $\Omega$.  Let $\mathring \rho(t)= e^{-\frac{1}{t}}$ for $0<t\leq\tfrac{1}{16}$, and $\mathring \rho(t)= 1$ for $t\geq \tfrac18$, with $\mathring \rho$ a smooth nondecreasing function supported on $t\geq 0$.  For $x\in \Omega$, let $d(x)$ be the Euclidean distance to the boundary $\partial \Omega$, i.e. $d(x)= \min (|x|-1, 2-|x|)$.   For $x\in \Omega$, let $\rho(x)= \mathring \rho(d(x))$.  Note that if $\mathring \nabla$ is the Euclidean connection, then for $d(x)<\tfrac{1}{16}$, $\mathring\nabla\rho(x)=( d(x))^{-2}\rho(x) \mathring \nabla d(x)$. By induction, we have that for $k\in \mathbb Z_+$, there is $C$ so that on $\Omega$, $|\mathring \nabla^k \rho|\leq C  d^{-2k} \rho$. 

We will let $0<\phi<d\leq 1$ be a smooth function on $\Omega$, which near the boundary $\partial \Omega$ (say for $d(x)<\tfrac18$) satisfies $\phi(x)= (d(x))^2$.  Note that for all $x\in \Omega$, the Euclidean ball $\overline{B_{\phi(x)}(x)}\subset \Omega$.  We note then that for each $k\in \mathbb Z_+$, there is a $C$ so that $|\phi^{k} \rho^{-1} \mathring \nabla^k \rho|\leq C$.  Moreover, there is a $C>1$ so that for all $x\in \Omega$ and $y\in B_{\phi(x)}(x)$, $C^{-1}\leq \tfrac{\phi(x)}{\phi(y)} \leq C$ and $C^{-1}\leq \tfrac{\rho(x)}{\rho(y)} \leq C$, since for $y\in B_{\phi(x)}(x)$, it follows that $d(x)-\phi(x) \leq d(y) \leq d(x)+ \phi(x)$. 

Given $k\in \mathbb Z_+$ and a metric $g$ on $\overline{\Omega}$, we define the weighted Sobolev space $H^k_\rho(\Omega,g)$ to be the space of functions (or sections of a tensor bundle) $u$ so that $|\nabla_g^j u|_g \in L^2(\Omega, \rho d\mu_g)$ for all $0\leq j\leq k$, with $\|u\|^2_{H^k_\rho(\Omega,g)}= \sum\limits_{ j\leq k} \int\limits_\Omega |\nabla_g^j u|^2_g \; \rho \;d\mu_g$.  We extend this to $k=0$, and let $L^2_\rho(\Omega, g)= H^0_{\rho}(\Omega,g)$.  $H^k_\rho(\Omega,g)$ has a natural Hilbert space structure.  For economy of notation we will sometimes suppress the metric: $H^k_\rho(\Omega):=H^k_\rho(\Omega,g)$.  We assume $g$ is smooth, or at least $C^k(\overline{\Omega})$, to define $\nabla_g^j u$ for a tensor field, with $j\leq k$.   We can interpret $\nabla_g^j u$ weakly, and we note that by \cite[Lemma 2.1]{cs:ak}, $C^{\infty}(\overline{\Omega})$ is dense in $H^k_\rho(\Omega)$.  We also note that we get an equivalent norm if we use the background Euclidean metric $\mathring g$.  

For $r, s\in \mathbb R$, let $\varphi= \phi^r \rho^s$.  For $\alpha\in (0,1]$ and $k$ a nonnegative integer, define the weighted H\"{o}lder space $C^{k,\alpha}_{\phi, \varphi}(\Omega)$ as the space of all $u\in C^{k,\alpha}_{\mathrm{loc}}(\Omega)$ for which
$$ \|u\|_{C^{k,\alpha}_{\phi, \varphi}(\Omega)}:= \sup\limits_{x\in \Omega} \big( \sum\limits_{j=0}^k \varphi(x) \phi^j(x) \|\mathring\nabla^j u\|_{C^{0}(B_{\phi(x)}(x))} +  \varphi(x) \phi^{k+\alpha} (x) [\mathring\nabla^k u]_{0,\alpha;B_{\phi(x)}(x) } \big)$$   is finite. $C^{k,\alpha}_{\phi, \varphi}(\Omega)$ is a Banach space.  We would obtain an equivalent norm  by using the connection $\nabla_g$ in place of the Euclidean connection, as in \cite[Appendix A]{cd}.  

We will also make use of the following spaces: $\mathcal B_0(\Omega)= C^{0,\alpha}_{\phi, \phi^{4+\frac{3}{2}} \rho^{-\frac{1}{2}}}(\Omega)\cap L^2_{\rho^{-1}}(\Omega)$, $\mathcal B_2(\Omega)= C^{2,\alpha}_{\phi, \phi^{2+\frac{3}{2}} \rho^{-\frac{1}{2}}}(\Omega)\cap L^2_{\rho^{-1}}(\Omega)$, while $\mathcal B_4(\Omega)= C^{4,\alpha}_{\phi, \phi^{\frac{3}{2}} \rho^{\frac{1}{2}}}(\Omega)\cap H^2_{\rho}(\Omega)$.  The norms on these spaces are defined by summing the relevant weighted Sobolev and H\"{o}lder norms, e.g. $\|u\|_{\mathcal B_0(\Omega)} :=\|u\|_{C^{0,\alpha}_{\phi, \phi^{4+\frac{3}{2}} \rho^{-\frac{1}{2}}}(\Omega)}+\|u\|_{L^2_{\rho^{-1}}(\Omega)}$.  As remarked above, changing the metric gives an equivalent norm, so we often suppress the metric.

\subsection{The Basic Weighted Injectivity Estimate} \label{sec:AWeightedEsts1}  We obtain a basic weighted coercivity estimate for $L_g^*$ for $g$ near $g_S^m$, in a suitable space transverse to $\mathrm{ker}(L_{g_S^m}^*)$ on $\Omega=\{ x\in \mathbb R^3:1<|x|<2\}$.

\begin{prop}  Let $\gamma$ be a metric in $C^0(\overline{\Omega})$.  Let $S_\gamma$ be the $L^2(\Omega,\gamma)$-orthogonal complement of $ \big( \zeta\, \mathrm{ker}(L_{g_S^m}^*)\big)$. There is a $C^{2}(\overline{\Omega})$-neighborhood $\mathcal U$ of $g_S^m$ and a $C>0$ so that for all $g\in \mathcal U$ and for all $u\in H^2_{\rho}(\Omega,g) \cap S_\gamma$, we have the following:
\begin{equation}\label{eq:b-est-wtd} \|u\|_{H^2_\rho(\Omega,g)}\leq C \|L_{g}^*u\|_{L^2_\rho(\Omega,g)} .  \end{equation} \label{lem:eq:b-est-wtd}
\end{prop}

\begin{proof}  Let $\mathring \varepsilon =\tfrac18$ and let $g\in \mathcal U$, where $\mathcal U$ is as in Lemma \ref{lem:b-est-0-eps} .  If suffices by density to establish the estimate for $u\in H^2(\Omega,g) \cap S_\gamma$.  Using the monotonicity of $\mathring \rho$, we have by (\ref{eq:b-est-0-eps}) that 
$$\int\limits_0^{\mathring \varepsilon} \mathring \rho'(\varepsilon) \|u\|^2_{H^2(\Omega_\varepsilon,g)} d\varepsilon \leq C^2 \int\limits_0^{\mathring \varepsilon}\mathring \rho'(\varepsilon) \|L_{g}^*u\|^2_{L^2(\Omega_\varepsilon,g)}d\varepsilon.$$  Applying the co-area formula $\int\limits_{\Omega\setminus \Omega_{\mathring \varepsilon}} f \; d\mu_g = \int\limits_0^{\mathring \varepsilon} \int\limits_{\{ x: d(x)=\varepsilon\}} f |\nabla_g d|_g^{-1}d\sigma_g \; d\varepsilon$ and integrating by parts, we have 
\begin{align*} \mathring \rho(\mathring \varepsilon) & \|u\|^2_{H^2(\Omega_{\mathring \varepsilon},g)} + \int\limits_0^{\mathring \varepsilon} \mathring \rho(\varepsilon)  \int\limits_{\{ x: d(x)=\varepsilon\}}\sum\limits_{j\leq 2} |\nabla_g^j u|^2_g|\nabla_g d|_g^{-1}d\sigma_g d\varepsilon \\ 
&\qquad \leq C^2 \Big( \mathring \rho(\mathring \varepsilon)  \|L_{g}^*u\|^2_{L^2(\Omega_{\mathring \varepsilon},g)} +\int\limits_0^{\mathring \varepsilon}\mathring \rho(\varepsilon) \int\limits_{\{ x: d(x)=\varepsilon\}} |L_{g}^*u|_g^2 |\nabla_g d|_g^{-1}d\sigma_g d\varepsilon \Big), 
\end{align*}
i.e. 
\begin{align*} \mathring \rho(\mathring \varepsilon)& \|u\|^2_{H^2(\Omega_{\mathring \varepsilon},g)}+   \|u\|^2_{H^2_{\rho}(\Omega\setminus \Omega_{\mathring \varepsilon},g)}  =\mathring \rho(\mathring \varepsilon)   \|u\|^2_{H^2(\Omega_{\mathring \varepsilon},g)} + \int\limits_{\Omega\setminus \Omega_{\mathring \varepsilon}} \sum\limits_{j\leq 2} |\nabla_g^j u|^2_g  \rho\; d\mu_g  \\ &\qquad \qquad \leq C^2 \Big( \mathring \rho(\mathring \varepsilon)  \|L_{g}^*u\|^2_{L^2(\Omega_{\mathring \varepsilon},g)} +\int\limits_{\Omega\setminus \Omega_{\mathring \varepsilon}}  |L_{g}^*u|_g^2 \rho \; d\mu_g \Big) \leq C^2   \|L_{g}^*u\|^2_{L^2_\rho(\Omega,g)} .
\end{align*}
Since $0<\mathring \rho\leq 1$, $\|u\|^2_{H^2_\rho(\Omega_{\mathring \varepsilon},g)} \leq \|u\|^2_{H^2(\Omega_{\mathring \varepsilon},g)} $, and we can conclude $$\|u\|^2_{H^2_{\rho}(\Omega,g)} \leq C^2 (\mathring \rho(\mathring \varepsilon))^{-1} \|L_{g}^*u\|^2_{L^2_\rho(\Omega,g)}.$$
\end{proof}

\begin{remark}  As remarked earlier, one could take a background metric such as $\mathring g$ to define the norms and obtain an analogous estimate. 
\end{remark}

\subsection{Weighted Schauder estimates} \label{sec:AWeightedEsts2} We record here a basic elliptic estimate in the weighted H\"{o}lder spaces.  The proof of the relevant estimate is a fairly straightforward scaling argument using the interior Schauder estimates, and can be found in \cite[Appendix A]{cem}, following \cite[Appendix B]{cd}, where more general operators are treated, cf. \cite[Appendix C]{cor-lan}.   We suppose $\phi$ and $\varphi$ are defined as above. 

\begin{prop} Suppose $P= \Delta^2  + \sum\limits_{|\beta|\leq 3} b_{\beta} \partial^{\beta}$ on $\Omega$, with $b_\beta\in C^{k, \alpha}_{\phi, \phi^{4-|\beta|}}(\Omega)$.  Then for $k$ a nonnegative integer and $0<\alpha<1$, there is a constant $C$ so that for all $u\in C^{k+4, \alpha}_{\phi,\varphi}(\Omega)$, 
\begin{equation} \label{eq:wtd-sch} \|u\|_{C^{k+4, \alpha}_{\phi,\varphi}(\Omega)} \leq C ( \| Pu\|_{C^{k, \alpha}_{\phi,\phi^4\varphi}(\Omega)}  + \|u\|_{L^2_{\phi^{-3} \varphi^2}(\Omega)}).  \end{equation} Moreover, if $u\in L^2_{\phi^{-3} \varphi^2}(\Omega)$ and $Pu\in C^{k, \alpha}_{\phi,\phi^4\varphi}(\Omega)$, then $u\in C^{k+4, \alpha}_{\phi,\varphi}(\Omega)$ and the above estimate holds.  \label{prop:wtd-sch}
\end{prop}

Of interest to us here, the operator $P=\tfrac{1}{2}\rho^{-1} L_g \rho L_g^*$ is of the above form, for $g\in C^{k+4, \alpha}(\overline{\Omega})$. 

\subsection{The linearized equation} \label{sec:ALinearized} We let $\zeta$ be as above, and let $S_g$ be the $L^2(\Omega,g)$-orthogonal complement of $\big( \zeta\, \mathrm{ker}(L_{g_S^m}^*)\big)$, and let $\Pi_g$ be the corresponding orthogonal projection onto $S_g$; let $\mathring S = S_{g_S^{m}}$ and $\mathring \Pi=\Pi_{g_S^{m}}$.  We will establish local deformations for the nonlinear operator $g\mapsto \mathring \Pi \circ R(g)$ for $g$ near $g_S^{m}$.  The linearization of this operator is $\mathring \Pi \circ L_g$, for which we want to establish surjectivity in appropriate spaces.  This will follow from the analogous result for the related operator $\Pi_g \circ L_g$.  

Let $\psi^\perp=\psi^{\perp_g}= \psi- \Pi_g(\psi)\in \big( \zeta\, \mathrm{ker}(L_{g_S^m}^*)\big)$.   Since $\zeta$ is supported in the set where $\rho=1$, we have that since $\Pi_g(\psi)$ and $\psi^\perp$ are orthogonal in $L^2(\Omega, g)$, they are also orthogonal in $L^2_{\rho^{-1}}(\Omega,g)$, hence $\|\psi\|^2_{L^2_{\rho^{-1}}(\Omega,g)} = \|\Pi_g(\psi)\|^2_{L^2_{\rho^{-1}}(\Omega,g)}+\|\psi^{\perp}\|^2_{L^2_{\rho^{-1}}(\Omega,g)}$.  We also note that there is a constant $C>0$ so that for all metrics $g$ near $g_S^{m}$, and all $\psi\in \mathcal B_0(\Omega)$ (using the fact that $\psi^{\perp}$ lies in a finite-dimensional space),
\begin{align*}
\|\Pi_g(\psi)\|_{\mathcal B_0(\Omega,g)} &\leq\|\psi\|_{\mathcal B_0(\Omega,g)}+\| \psi^\perp\|_{\mathcal B_0(\Omega,g)}\leq  \|\psi\|_{\mathcal B_0(\Omega,g)}+ C\| \psi^\perp\|_{L^2_{\rho^{-1}}(\Omega,g)}\\
&\leq  \|\psi\|_{\mathcal B_0(\Omega,g)}+ C\| \psi\|_{L^2_{\rho^{-1}}(\Omega,g)} \leq (1+C) \|\psi\|_{\mathcal B_0(\Omega,g)}.
\end{align*}

We also note that applying $\Pi_g$ to $\psi-\mathring \Pi (\psi)$ we conclude 
$$\|\Pi_g(\psi)\|_{L^2_{\rho^{-1}(\Omega,g)} }= \|\Pi_g(\mathring \Pi (\psi))\|_{L^2_{\rho^{-1}(\Omega,g)}}\leq \|\mathring \Pi (\psi)\|_{L^2_{\rho^{-1}(\Omega,g)}}\leq C \|\mathring \Pi(\psi)\|_{L^2_{\rho^{-1}(\Omega,g_S^{m_0})}}$$
where $C$ is uniform for $g$ near $g_S^{m}$.

Now, if $u\in C^\infty_c(\Omega)$, then in the $L^2(\Omega, g)$-orthogonal decomposition $u=u_0+u_1$ where $u_0\in \big( \zeta\, \mathrm{ker}(L_{g_S^m}^*)\big)$ and $u_1\in S_g$, we have $u_1\in C^\infty_c(\Omega)$ as well.   To say that $ \Pi_g ( L_g \rho L_g^*(f)) =\Pi_g (\psi)$ weakly i.e. $ L_g \rho L_g^*(f)-\psi\in \big( \zeta\, \mathrm{ker}(L_{g_S^m}^*)\big)$ weakly, can then be interpreted as $ \int\limits_\Omega ( \rho L_g^*(u) \cdot_g L_g^* (f) - \psi u )\; d\mu_g=0$ for all $u\in C^\infty_c(\Omega) \cap S_g$.

\begin{prop} There is a $C^{4,\alpha}(\overline{\Omega})$-neighborhood $\mathcal U$ of $g_S^{m}$ along with a constant $C_1>0$ so that for all $g\in \mathcal U$ and for $\psi\in L^2_{\rho^{-1}}(\Omega)$, there is a unique $f\in H^2_{\rho}(\Omega,g)\cap S_g$ that weakly solves
\begin{equation} \Pi_g ( L_g \rho L_g^*(f)) =\Pi_g (\psi),  
\end{equation}
along with the estimate \begin{equation} 
\|f\|_{H^2_{\rho}(\Omega,g)}\leq C_1 \|\Pi_g (\psi)\|_{L^2_{\rho^{-1}}(\Omega,g)}.
\end{equation}
For such $f$ we have $\mathring \Pi ( L_g \rho L_g^*(f)) =\mathring \Pi (\psi)$, with $\|f\|_{H^2_{\rho}(\Omega,g)}\leq C_1 \|\mathring \Pi (\psi)\|_{L^2_{\rho^{-1}}(\Omega,g)}.$
\end{prop} 

\begin{proof}  Given $\psi \in L^2_{\rho^{-1}}(\Omega,g)$, let $\mathcal G(u)$ be defined for $u\in H^2_{\rho}(\Omega,g)\cap S_g$ by 
$$\mathcal G(u)= \int\limits_{\Omega} \Big( \tfrac12 \rho |L_g^*(u)|^2_g-\Pi_g(\psi) u \Big) d\mu_g= \int\limits_{\Omega} \Big( \tfrac12 \rho |L_g^*(u)|^2_g-\psi u \Big) d\mu_g.$$  By (\ref{eq:b-est-wtd}), we have 
$$ \int\limits_{\Omega} \Big( \tfrac12 |L_g^*(u)|^2_g-\Pi_g(\psi) u\Big) d\mu_g \geq C \|u\|_{H^2_{\rho}(\Omega,g)}^2 - \|\Pi_g(\psi) \|_{L^2_{\rho^{-1}}(\Omega, g)} \cdot \|u\|_{L^2_{\rho}(\Omega, g)}.$$  Since $\zeta$ has compact support, $S_g$ is closed in $H^2_{\rho}(\Omega,g)\cap S_g$.  A standard argument using Banach-Alaoglu and Riesz Representation for Hilbert spaces as in \cite[p. 150-152]{cor:schw} provides a minimizer $f\in H^2_{\rho}(\Omega,g)\cap S_g$.  Such a minimizer is unique by the strict convexity: if $u_1\neq u_2$ are two elements in $H^2_{\rho}(\Omega,g)\cap S_g$, then for $0<s<1$, $\mathcal G(su_1+(1-s)u_2)> s\mathcal G(u_1) + (1-s) \mathcal G(u_2)$, an easy estimate using the arithmetic-geometric mean inequality, and the injectivity of $L_g^*$ on $H^2_{\rho}(\Omega,g)\cap S_g$.  Moreover since for the minimizer $\mathcal G(f)\leq 0$, the desired estimate holds. 

We now consider the Euler-Lagrange equations.  For $u\in S_g\cap C^\infty_c(\Omega)$, 
\begin{align*}
0&= \frac{d}{dt}\Big|_{t=0} \mathcal G(f+tu) = \int\limits_\Omega \Big( \rho L_g^*(u) \cdot_g L_g^* (f) - \Pi_g(\psi) u \Big)\; d\mu_g\\
& =\int\limits_\Omega \Big(\rho L_g^*(u) \cdot_g L_g^* (f) - \psi u \Big) \; d\mu_g.
\end{align*} 
\end{proof} 

\begin{remark} Let $T$ be the parity map $T(x)=-x$.  If $g$ and  $\psi$ are parity-symmetric (i.e. $T^* g= g$ and $ \psi= \psi\circ T$ on $\Omega$), then the minimizer $f$ inherits the symmetry, since the minimizer is unique and $\mathcal G(f) = \mathcal G(f\circ T)$ in this case. \label{rmk:psym}
\end{remark}

Since the principal part of $L_g L^*_g$ is $2\Delta_g^2$, elliptic regularity immediately gives us that $f$ in the preceding is in $H^4_{\mathrm{loc}}(\Omega)$, and for smooth $g$, if $\psi$ is smooth, then $f$ is smooth in $\Omega$ too.   For suitable $\psi$ we want to show that $\rho L_g^* f$ extends smoothly by zero over the boundary of $\Omega$.  For that we have the following proposition.  We suppress the metric notation in the norms, or for definiteness, we could just use the connection $\mathring \nabla$ for $\mathring g$ and the measure $d\mu_{\mathring g}$ in defining the norms (similarly for any fixed (sufficiently) smooth metric). 

\begin{prop} There is a $C^{4,\alpha}(\overline{\Omega})$-neighborhood $\mathcal U$ of $g_S^{m}$ along with a constant $C_2>0$ so that for all $g\in \mathcal U$ and for $\psi\in \mathcal B_0(\Omega)$, if $f\in H^2_{\rho}(\Omega)\cap S_g$ weakly solves
\begin{equation} \mathring \Pi ( L_g \rho L_g^*(f)) =\mathring \Pi (\psi),  
\end{equation}
then $f\in \mathcal B_4(\Omega)$ and 
\begin{equation} 
\|f\|_{\mathcal B_4(\Omega)}\leq C_2 \|\mathring \Pi (\psi)\|_{\mathcal B_0(\Omega)}. \label{eq:wtd-f}
\end{equation}
Moreover if we let $h=\rho L_g^* f$, then 
\begin{equation} 
\|h\|_{\mathcal B_2(\Omega)}\leq C_2 \|\mathring \Pi (\psi)\|_{\mathcal B_0(\Omega)}. \label{eq:wtd-h}
\end{equation}
\end{prop}

\begin{proof}  Let $P= \rho^{-1} L_g \rho L_g^*$.  Since $\rho=1$ on the support of $\zeta$, $u^{\perp}= \rho^{-1}(\rho u)^{\perp}$.  Then we see $P(f)= \rho^{-1} \Big(\mathring \Pi (\rho P(f)) + (\rho P(f))^{\perp}\Big)=\rho^{-1} \mathring \Pi (\psi) + (P(f))^{\perp}.$ We apply (\ref{eq:wtd-sch}) to obtain (we omit the domain $\Omega$ in the notation for the norms)
\begin{align*}
\|f\|_{C^{4,\alpha}_{\phi, \phi^{\frac{3}{2}}\rho^{\frac{1}{2}} }} & \leq C\left( \|P(f)\|_{C^{0,\alpha}_{\phi, \phi^{4+\frac{3}{2}}\rho^{\frac{1}{2}}}}+\|f\|_{L^2_{\rho}}\right)\\
&\leq C\left(  \|\rho^{-1} \mathring \Pi(\psi)\|_{C^{0,\alpha}_{\phi, \phi^{4+\frac{3}{2}}\rho^{\frac{1}{2}}}}+\|P(f)^{\perp}\|_{C^{0,\alpha}_{\phi, \phi^{4+\frac{3}{2}}\rho^{\frac{1}{2}}}}+\|f\|_{L^2_{\rho}}\right).
\end{align*}
The last term on the right can be estimated by the preceding proposition, and for the first term we have $\|\rho^{-1} \mathring \Pi(\psi)\|_{C^{0,\alpha}_{\phi, \phi^{4+\frac{3}{2}}\rho^{\frac{1}{2}}}}\leq C \|\mathring \Pi(\psi)\|_{C^{0,\alpha}_{\phi, \phi^{4+\frac{3}{2}}\rho^{-\frac{1}{2}}}}.$  Thus we only need to estimate $\|P(f)^{\perp}\|_{C^{0,\alpha}_{\phi, \phi^{4+\frac{3}{2}}\rho^{\frac{1}{2}}}}$.  Since $P(f)^{\perp}$ lies in a finite-dimensional space, there is a $C>0$ so that for all $f$, $\|P(f)^{\perp}\|_{C^{0,\alpha}_{\phi, \phi^{4+\frac{3}{2}}\rho^{\frac{1}{2}}}}\leq C \|P(f)^{\perp}\|_{L^2(\mathrm{spt}(\zeta))}$.  On the other hand, by interpolation, for any $\epsilon>0$, there is $C(\epsilon)$ so that 
\begin{align*} \|P(f)^{\perp}\|_{L^2(\mathrm{spt}(\zeta))} &\leq  \|f\|_{C^4(\mathrm{spt}(\zeta))}\leq \epsilon  \|f\|_{C^{4,\alpha}(\mathrm{spt}(\zeta))}+ C(\epsilon)  \|f\|_{L^2(\mathrm{spt}(\zeta))}\\
&\leq C \left(  \epsilon  \|f\|_{C^{4,\alpha}_{\phi, \phi^{\frac{3}{2}}\rho^{\frac{1}{2}} }}+ C(\epsilon)  \|f\|_{L^2_{\rho}}\right) .
\end{align*} 

As for $h=\rho L_{g}^* (f)$, $\|h\|_{L^2_{\rho^{-1}}} \leq C \| f\|_{L^2_{\rho}}$, while for some $C_0$
\begin{align*}
\|h\|_{C^{2,\alpha}_{\phi, \phi^{2+\frac{3}{2}} \rho^{-\frac{1}{2}}} } =\|\rho L_{\gamma}^*(f)\|_{C^{2,\alpha}_{\phi, \phi^{2+\frac{3}{2}}\rho^{-\frac{1}{2}}} }\leq C_0 \|f\|_{C^{4,\alpha}_{\phi, \phi^{\frac{3}{2}}\rho^{\frac{1}{2}}} }.
\end{align*}
\end{proof}

\subsection{Solving the nonlinear problem by iteration} \label{sec:ANonlinear}  We continue the proof of Proposition \ref{prop:def0}.  Given $\gamma$ close to $g_S^{m}$ we want to solve for suitable $h(\gamma)$ with $\mathring \Pi (R(\gamma+h(\gamma)))=0$.  In fact we find $h(\gamma)$ in the form $h(\gamma)=\rho L_{\gamma}^*(f)$, for suitable $f$.  We do this iteratively, following the same framework as our previous works \cite{cor:schw, cs:ak, cem, cor-lan}, cf. \cite{cd}.   

Let $\mathscr L_{\gamma}= L_{\gamma}\rho L_{\gamma}^*$. We recursively define $f_j$, $h_j$ and $\gamma_j$ as follows.  Let $\gamma_0=\gamma$.  Let $f_0$ solve $\mathring \Pi (\mathscr L_{\gamma}(f_0))= -\mathring \Pi (R(\gamma))$.  Let $h_0= \rho L_{\gamma}^*(f_0)$.  For $\gamma$ close enough to $g_S^{m}$, $R(\gamma)$ is sufficiently small so that $\gamma_1=\gamma_0+h_0$ is a small perturbation of $\gamma$, and hence is a metric close to $g_S^{m}$.   We next solve $\mathring \Pi (\mathscr L_{\gamma}(f_1))= -\mathring \Pi (R(\gamma_1))$; we then let $h_1=\rho L_{\gamma}^*(f_1)$, and $\gamma_2=\gamma_1+h_1$.   We then justify $h_1$ is small, and $\gamma_2$ is a metric.   We remark that we linearize at a fixed metric, to control some estimates required to close the argument; to illustrate, the operator $L_{\gamma_1}$ involves derivatives of the tensor $h_0$ we generated, whereas we have fixed estimates on $L_\gamma$.  As such, though, the convergence to a limit may be slower than that of Newton's method, since the improvement is sub-quadratic after the first iteration.

\begin{lemma} Suppose $0<\delta<1$.  There is a constant $C$ and a $C^{4,\alpha}(\overline{\Omega})$-neighborhood $\mathcal U$ of $g_S^{m}$ so that the recursion outlined above produces infinite sequences $f_j$ and $h_j=\rho L_{\gamma}^*(f_j)$ and $\gamma_{j}$ with $\gamma_0=\gamma$, and $\gamma_{j+1}= \gamma_0+ \sum\limits_{k=0}^{j} h_k$ where each $\gamma_j$ is a metric, with the following estimates:
\begin{align}
\|f_j\|_{\mathcal B_4(\Omega)}&\leq C \|\mathring \Pi (R(\gamma))\|_{\mathcal B_0(\Omega)}^{(1+j)\delta} \label{eq:fjest}\\
\|h_j\|_{\mathcal B_2(\Omega)}&\leq C \|\mathring \Pi (R(\gamma))\|_{\mathcal B_0(\Omega)}^{(1+j)\delta} \label{eq:hjest} \\
\|\mathring \Pi(R(\gamma_j))\|_{\mathcal B_0(\Omega)} &\leq \|\mathring \Pi(R(\gamma))\|_{\mathcal B_0(\Omega)}^{(1+j)\delta}. \label{eq:rjest}
\end{align} \label{lem:eq:rjest}
\end{lemma} 

\begin{proof}
The proof follows along the same lines as proofs of the analogous results in earlier works, in particular \cite[Lemma 3.5]{cem}, cf. \cite[Sec. 6.3]{cor-lan}, so we just indicate the ideas.   The required estimate (\ref{eq:fjest}) for $f_0$ is given by (\ref{eq:wtd-f}), and likewise for $h_0$ we have (\ref{eq:wtd-h}).  The required estimate (\ref{eq:rjest}) for $\mathring \Pi(R(\gamma_1))$ follows fairly readily from estimating the quadratic Taylor remainder for the scalar curvature functional $R(\gamma_1)= R(\gamma)+ L_{\gamma} (h_0)+ \mathcal Q_{\gamma}(h_0)$ (cf. \cite[Sec. 3.5]{cem}).  Note that for $j=1$ we could allow $\delta=1$, but in the end we use $\delta<1$, as we get slower convergence due to using the linearization at $\gamma=\gamma_0$ only, as we illustrate now. 

Feeding the estimate (\ref{eq:rjest}) for $\mathring \Pi(R(\gamma_1))$ into (\ref{eq:wtd-f}) gives the required estimate (\ref{eq:fjest}) of $f_1$, and then the required estimate (\ref{eq:hjest}) of $h_1$ follows likewise.  As for the estimate (\ref{eq:rjest}) for $j=2$ and $j=3$, we write \begin{align*}
\mathring \Pi(R(\gamma_2))&= \mathring \Pi \Big( R(\gamma_1)  + L_{\gamma}(h_1)+ (L_{\gamma_1}-L_{\gamma})(h_1) + \mathcal Q_{\gamma_1}(h_1)\Big)\\
&=\mathring \Pi \Big( (L_{\gamma_1}-L_{\gamma})(h_1) + \mathcal Q_{\gamma_1}(h_1)\Big)\\
\mathring \Pi(R(\gamma_3))&= \mathring \Pi \Big( R(\gamma_2)  + L_{\gamma}(h_2)+ (L_{\gamma_2}-L_{\gamma_1})(h_2) + [(L_{\gamma_1}-L_{\gamma})(h_2)+ \mathcal Q_{\gamma_2}(h_2)\Big)\\
&=\mathring \Pi \Big( (L_{\gamma_1}-L_{\gamma})(h_1) + \mathcal Q_{\gamma_1}(h_1)\Big).
\end{align*}
As $\gamma_1-\gamma=h_0$, and $\gamma_2-\gamma_1=h_1$, $\|\mathring \Pi(R(\gamma_2))\|_{\mathcal B_0(\Omega)}$ and $\|\mathring \Pi(R(\gamma_3))\|_{\mathcal B_0(\Omega)}$ above are of quadratic order in $(h_0, h_1, h_2)$.  We also note that since we control $h_j$ in $C^{2,\alpha}(\overline{\Omega})$, there is a uniform estimate of $\|\mathcal Q_{\gamma_j}(h_j)\|_{\mathcal B_0(\Omega)}$.  From here it should be clear how to proceed, and we refer to \cite[Lemma 3.5]{cem}, cf. \cite[Sec. 6.3]{cor-lan}, for more details. 
\end{proof}

It follows from this lemma that the iteration converges and yields a solution $h=h(\gamma)= \rho L_\gamma^* f$ as desired.  Moreover we can conclude that the extension of $h$ by zero across $\partial \Omega$ is in $C^{2,\alpha}(\mathbb R^3)$, and using Remark \ref{rmk:psym} at each stage of the iteration, we have that if $\gamma$ is parity symmetric, so is $h(\gamma)$.  It remains to show that $h$ is smooth, and that $h(\gamma)$ depends continuously on $\gamma$.  

The operator $u\mapsto P(u):=\rho^{-1} R(\gamma+\rho L_{\gamma}^*(u)) - R(\gamma))$ is a quasilinear fourth-order elliptic operator provided $\rho L_\gamma^*(u)$ is sufficiently small, in which case the fourth-order part of the operator is close to that of $\rho^{-1} L_\gamma \rho L_{\gamma}^*$ (i.e. $2 \Delta_\gamma^2$).  So in case $\gamma$ is smooth, $h=h(\gamma)$ is smooth in $\Omega$.  To see that it extends smoothly by zero over the boundary $\partial \Omega$, we consider the equation $R(\gamma+h(\gamma))-R(\gamma)=b(\gamma) \zeta f^{m}-R(\gamma)$; as the right-hand side is supported outside a collar neighborhood of the boundary, we can apply Proposition \ref{prop:wtd-sch} for any nonnegative integer $k$, from which we can conclude smoothness, and the required higher-order estimates. 

\begin{remark} The support of $h(\gamma)$ is contained in $\overline{\Omega}$.  We could of course have arranged the support to lie strictly inside $\Omega$ by running the construction on $\Omega_\varepsilon$ in place of $\Omega$, for some small $\varepsilon>0$. \end{remark}

\subsection{Continuous dependence} \label{ssec:cdep}

Finally we prove continuous dependence of $h(\gamma)$ and hence $b(\gamma)$ on $\gamma$.  

\begin{proof}  Suppose we have $\gamma_1$ and $\gamma_2$ near $g_S^{m}$, with associated $h_1=h(\gamma_1)=\rho L^*_{\gamma_1}(f_1)$ and $h_2=h(\gamma_2)=\rho L^*_{\gamma_2}(f_2)$, so that $\mathring \Pi(R(\gamma_i+h_i))=0$.  Then writing  
\begin{align*}
L_{\gamma_1}(h_2-h_1)&=L_{\gamma_2}(h_2)-L_{\gamma_1}(h_1) + (L_{\gamma_1}-L_{\gamma_2})(h_2)
 \end{align*}
 and applying a Taylor expansion $R(\gamma_i+h_i)=R(\gamma_i)+ L_{\gamma_i} (h_i) + \mathcal Q_{\gamma_i}(h_i)$, we see 
 \begin{align*}
 L_{\gamma_1}(h_2-h_1)&=R(\gamma_2+h_2) -R(\gamma_1+h_1)+R(\gamma_1)-R(\gamma_2) + \mathcal Q_{\gamma_1}(h_1) - \mathcal Q_{\gamma_2}(h_2)\\ 
 &\qquad + (L_{\gamma_1}-L_{\gamma_2})(h_2)\\ 
 &=(b(\gamma_2)-b(\gamma_1)) \zeta f^{m}+R(\gamma_1)-R(\gamma_2) + \mathcal Q_{\gamma_1}(h_1) - \mathcal Q_{\gamma_2}(h_2)\\ 
 &\qquad + (L_{\gamma_1}-L_{\gamma_2})(h_2) .
 \end{align*}

Based on the form of the Taylor remainder, there is a constant $C>0$ so that 
\begin{align} \|\mathcal Q_{\gamma_1}(h_1) & - \mathcal Q_{\gamma_2}(h_2)\|_{\mathcal B_0(\Omega)} \nonumber \\ & \leq C \|( h_1,h_2)\|_{\mathcal B_2(\Omega)} \Big(  \|\gamma_1-\gamma_2\|_{C^{2,\alpha}(\Omega)} \|( h_1,h_2)\|_{\mathcal B_2(\Omega)} 
+  \|h_1-h_2\|_{\mathcal B_2(\Omega)}\Big). \label{eq:Q-est} \end{align}

 We will want to express everything in terms of $f_1$ and $f_2$.  Note that 
  \begin{align}
 L_{\gamma_1} [\rho L_{\gamma_1}^* (f_2-f_1) ]&= L_{\gamma_1}(h_2-h_1) + L_{\gamma_1} ( \rho (L_{\gamma_1}^*-L_{\gamma_2}^*)(f_2)) \nonumber \\
 &=(b(\gamma_2)-b(\gamma_1)) \zeta f^{m}+R(\gamma_1)-R(\gamma_2) + \mathcal Q_{\gamma_1}(h_1) - \mathcal Q_{\gamma_2}(h_2) \nonumber \\ 
 &\qquad + (L_{\gamma_1}-L_{\gamma_2})(h_2)+L_{\gamma_1} [\rho (L_{\gamma_1}^*-L_{\gamma_2}^*)(f_2)]\nonumber \\
  &=:(b(\gamma_2)-b(\gamma_1)) \zeta f^{m}+\mathcal E_1+ L_{\gamma_1} [\rho (L_{\gamma_1}^*-L_{\gamma_2}^*)(f_2)]\nonumber\\
 &=: (b(\gamma_2)-b(\gamma_1)) \zeta f^{m}+ \mathscr E_1+\mathcal E_2=:(b(\gamma_2)-b(\gamma_1)) \zeta f^{m}+ \mathscr E.\label{eq:f-diff} 
 \end{align}
 We estimate each of these terms.  By (\ref{eq:Q-est}), and the fact that $R(\gamma_1)-R(\gamma_2)$ is supported in $\overline{\Omega'}$, away from $\partial \Omega$, we see for some $C>0$,
 \begin{align} \| \mathcal E_1\|_{\mathcal B_0(\Omega)}& \leq C\Big(  \|\gamma_1-\gamma_2\|_{C^{2,\alpha}(\Omega_0)} + \|h_1-h_2\|_{\mathcal B_2(\Omega)}\|(h_1,h_2)\|_{\mathcal B_2(\Omega)} \Big) \label{eq:err1} \\
   \| \mathcal E_2\|_{C^{0,\alpha}_{\varphi, \varphi^{4+\frac32} \rho^{-\frac12}}(\Omega)} &\leq C \|\gamma_1-\gamma_2\|_{C^{4,\alpha}(\Omega)} \|f_2\|_{C^{4,\alpha}_{\varphi, \varphi^{\frac32} \rho^{\frac12}}(\Omega)} \nonumber \\ & \leq C \|\gamma_1-\gamma_2\|_{C^{4,\alpha}(\Omega)} \|f_2\|_{\mathcal B_4(\Omega)}. \label{eq:err2} 
 \end{align}
Later on we will also want to estimate $\int\limits_\Omega (f_2-f_1) \mathscr E d\mu_{\gamma_1}$, for which we note that, using the density of $C^{\infty}_c(\Omega)$ in $H^2_\rho(\Omega)$, 
 \begin{align}
  \Big| \int\limits_\Omega (f_2-f_1) \mathscr E_1 d\mu_{\gamma_1}\Big| &\leq \|f_2-f_1\|_{L^2_\rho(\Omega, \gamma_1)}\|\mathscr E_1\|_{L^2_{\rho^{-1}}(\Omega, \gamma_1)}\nonumber \\ & \leq C \|f_2-f_1\|_{L^2_\rho(\Omega, \gamma_1)}\|\mathscr E_1\|_{\mathcal B_0(\Omega)} \label{eq:int-err1} \\
    \Big| \int\limits_\Omega (f_2-f_1) \mathscr E_2 d\mu_{\gamma_1}\Big|&=\Big| \int\limits_\Omega \rho L_{\gamma_1}^*(f_2-f_1)\cdot_{\gamma_1}   [(L_{\gamma_1}^*-L_{\gamma_2}^*)(f_2)] d\mu_{\gamma_1}\Big| \nonumber\\
    & \leq  C  \|\gamma_1-\gamma_2\|_{C^{2}(\Omega)} \|f_2\|_{H^2_\rho(\Omega,\gamma_1)} \|f_2-f_1\|_{H^2_\rho(\Omega, \gamma_1)}. \label{eq:int-err2}
 \end{align}

We next estimate $(b(\gamma_2)-b(\gamma_1)) \zeta f^{m}= R(\gamma_2+h_2)-R(\gamma_1+h_1)$.  For simplicity of exposition, consider the case $m>0$.   Applying the Taylor expansion we have 
\begin{align*}
(b(\gamma_2)-b(\gamma_1))  \int\limits_\Omega \zeta (f^{m})^2 d\mu_{\gamma_1}  &= \int\limits_{\Omega} (R(\gamma_2)-R(\gamma_1)) f^{m} d\mu_{\gamma_1} + \int\limits_{\Omega} L_{\gamma_2}(h_2)  f^{m}d\mu_{\gamma_2} \nonumber \\
&\quad + \int\limits_{\Omega} L_{\gamma_2}(h_2 )f^{m}[ d\mu_{\gamma_1}-d\mu_{\gamma_2}]  -\int\limits_{\Omega} L_{\gamma_1} (h_1) f^{m}d\mu_{\gamma_1}\\
&\quad +\int\limits_{\Omega} [\mathcal Q_{\gamma_2} (h_2)-\mathcal Q_{\gamma_1} (h_1)] f^{m}d\mu_{\gamma_1}.\nonumber
\end{align*}
Therefore by (\ref{eq:Q-est})
\begin{align} |b(\gamma_2)-b(\gamma_1)| & \leq C\left(  \| \gamma_1-\gamma_2\|_{C^{2,\alpha}(\Omega)} + \|h_1-h_2\|_{\mathcal B_2(\Omega)} \|(h_1, h_2)\|_{\mathcal B_2(\Omega)} \right) \nonumber \\
& \qquad + \left|  \int\limits_{\Omega} L_{\gamma_2}(h_2)  f^{m}d\mu_{\gamma_2} -\int\limits_{\Omega} L_{\gamma_1} (h_1) f^{m}d\mu_{\gamma_1}\right| .  \label{eq:c-est-0} \end{align}

We just need to estimate the integrals in (\ref{eq:c-est-0}).  This is easy, but the required expansion is cumbersome.   Since $L_{g_S^{m}}^* (f^{m})=0$,  and $h(\gamma_i)$ vanishes along with derivatives at $\partial \Omega$, we obtain

\begin{align*} 
\int\limits_{\Omega} & L_{\gamma_2}(h_2  ) f^{m}d\mu_{\gamma_2}  -\int\limits_{\Omega} L_{\gamma_1}(h_1)   f^{m}d\mu_{\gamma_1} \\
&=\int\limits_{\Omega} L_{\gamma_2}(h_2)   f^{m}[d\mu_{\gamma_2}-d\mu_{g_S^{m}}]+ \int\limits_{\Omega} [L_{\gamma_2}-L_{g_S^{m}}] (h_2 ) f^{m}d\mu_{g_S^{m}} -\int\limits_{\Omega} L_{\gamma_1}(h_1)   f^{m}d\mu_{\gamma_1}\\
&=  \int\limits_{\Omega} L_{\gamma_2}(h_2 -h_1)  f^{m}[d\mu_{\gamma_2}-d\mu_{g_S^{m}}] +\int\limits_{\Omega} [L_{\gamma_2}-L_{\gamma_1}] (h_1)  f^{m}[d\mu_{\gamma_2}-d\mu_{g_S^{m}}]\\
& \qquad +\int\limits_{\Omega}  L_{\gamma_1}(h_1 ) f^{m}[d\mu_{\gamma_2} -d\mu_{\gamma_1}] + \int\limits_{\Omega} [L_{g_S^{m}}-L_{\gamma_1} ] (h_1) f^{m}d\mu_{g_S^{m}} \\
&\qquad  + \int\limits_{\Omega} [L_{\gamma_2}-L_{g_S^{m}}] (h_2)   f^{m}d\mu_{g_S^{m}} \\
&=  \int\limits_{\Omega} L_{\gamma_2}(h_2 -h_1)  f^{m}[d\mu_{\gamma_2}-d\mu_{g_S^{m}}] +\int\limits_{\Omega} [L_{\gamma_2}-L_{\gamma_1}] (h_1)  f^{m}[d\mu_{\gamma_2}-d\mu_{g_S^{m}}]\\
& \qquad +\int\limits_{\Omega}  L_{\gamma_1}(h_1)   f^{m}[d\mu_{\gamma_2} -d\mu_{\gamma_1}] + \int\limits_{\Omega} [L_{\gamma_2}-L_{\gamma_1} ] (h_1)   f^{m}d\mu_{g_S^{m}} \\
&\qquad  + \int\limits_{\Omega}[ L_{\gamma_2}-L_{g_S^{m}}] (h_2 -h_1) f^{m}d\mu_{g_S^{m}}. 
\end{align*}
 Together with (\ref{eq:c-est-0}) this allows us to conclude 
 \begin{align} |b(\gamma_2)-b(\gamma_1)| & \leq C \| \gamma_1-\gamma_2\|_{C^{2,\alpha}(\Omega)} \nonumber\\ & \qquad + D \|h_1-h_2\|_{\mathcal B_2(\Omega)} \left( \|\gamma_2-g_S^{m}\|_{C^2(\Omega)} + \|(h_1, h_2)\|_{\mathcal B_2(\Omega)} \right)  . \label{eq:c-est-1} \end{align}
 We obtain the same estimate in the $m=0$ case by integrating $b(\gamma_2)-b(\gamma_1)) \zeta f^{0}= R(\gamma_2+h_2)-R(\gamma_1+h_1)$ against each of the components of $f^0=\langle 1, x^1, x^2, x^3\rangle$, and estimating just as above. 
 
Next, we estimate $\|f_2-f_2\|_{L^2_\rho(\Omega, \gamma_1)}$.  To this end, multiply (\ref{eq:f-diff}) by $(f_2-f_1)$ and integrate,  using that $f_i\in S_{\gamma_i}$ along with the density of $C^{\infty}_c(\Omega)$ in $H^2_\rho(\Omega)$, to obtain 

 \begin{align}\|L_{\gamma_1}^*(f_2-f_1)\|^2_{L^2_{\rho}(\Omega, \gamma_1)} &= \int\limits_\Omega (f_2-f_1)(b(\gamma_2)-b(\gamma_1)) \zeta f^{m}d\mu_{\gamma_1} + \int\limits_\Omega (f_2-f_1) \mathscr E d\mu_{\gamma_1}
 \label{eq:cty-est-0}. 
 \end{align}
 
We use this to estimate $f_2-f_1$.  Observe that if $L^*_{\gamma_1}$ has trivial kernel, then by the same arguments in the proof of Proposition \ref{lem:eq:b-est-wtd}, there is a constant $C_1$ so that for all $f\in H^2_\rho(\Omega,\gamma_1)$, $\|f\|_{H^2_{\rho}(\Omega,\gamma_1)}\leq C_1 \|L_{\gamma_1}^* (f)\|_{L^2_\rho(\Omega, \gamma_1)}$, which gives us an estimate by setting $f=f_2-f_1$. In general, $f_2=\Pi_{\gamma_1}(f_2)+\widetilde f_2$, where for $m>0$, 
 \begin{align} \label{eq:f2} \widetilde f_2 = \tfrac{\langle f_2, \zeta f^{m}\rangle_{L^2(\Omega, \gamma_1)}}{\langle\zeta f^{m}, \zeta f^{m}\rangle_{L^2(\Omega, \gamma_1)}} \; \zeta f^{m}=\tfrac{\int\limits_\Omega f_2  \zeta f^{m} (d\mu_{\gamma_1}-d\mu_{\gamma_2})}{\langle\zeta f^{m}, \zeta f^{m}\rangle_{L^2(\Omega, \gamma_1)}} \; \zeta f^{m}\; , \end{align} and analogously for $m=0$.  Hence $\widetilde f_2= a \zeta f^m$, where $\|a\|\leq C \| \gamma_2-\gamma_1\|_{C^0(\Omega)}$. Using this and (\ref{eq:b-est-wtd}) we have 
 \begin{align}
 \|f_2-f_1\|_{H^2_\rho(\Omega, \gamma_1)} &\leq \|\Pi_{\gamma_1} (f_2)-f_1\|_{H^2_\rho(\Omega, \gamma_1)}+ \|\widetilde f_2\|_{H^2_\rho(\Omega, \gamma_1)} \nonumber \\
 & \leq C \|L^*_{\gamma_1}[\Pi_{\gamma_1} (f_2)-f_1]\|_{L^2_\rho(\Omega, \gamma_1)}+\|\widetilde f_2\|_{H^2_\rho(\Omega, \gamma_1)}  \nonumber\\
 &\leq C \|L^*_{\gamma_1}(f_2-f_1)\|_{L^2_\rho(\Omega, \gamma_1)} +C \|L^*_{\gamma_1}(\widetilde f_2)\|_{L^2_\rho(\Omega, \gamma_1)}+\|\widetilde f_2\|_{H^2_\rho(\Omega, \gamma_1)} \nonumber\\
  &\leq C \|L^*_{\gamma_1}(f_2-f_1)\|_{L^2_\rho(\Omega, \gamma_1)} +C' \|\widetilde f_2\|_{H^2_\rho(\Omega, \gamma_1)} \nonumber\\
  &\leq C_1 \Big(\|L^*_{\gamma_1}(f_2-f_1)\|_{L^2_\rho(\Omega, \gamma_1)} +\|\gamma_2-\gamma_1\|_{C^0(\Omega)} \Big).
 \end{align}

 We thus conclude, using (\ref{eq:cty-est-0}) together with (\ref{eq:err1}), (\ref{eq:int-err1}), (\ref{eq:int-err2}) and (\ref{eq:c-est-1}), 
  \begin{align*}
 \|f_2  -f_1\|^2_{H^2_\rho(\Omega, \gamma_1)} & \leq 2(C_1)^2 \Big(\|L^*_{\gamma_1}(f_2-f_1)\|^2_{L^2_\rho(\Omega, \gamma_1)} + \|\gamma_2-\gamma_1\|^2_{C^0(\Omega)}\Big) \\
 &\leq C_2 \|f_2-f_1\|_{H^2_\rho(\Omega, \gamma_1)} \Big( \|\gamma_2-\gamma_1\|_{C^{2,\alpha}(\Omega)} \\
 & \qquad + \|h_2-h_2\|_{\mathcal B_2(\Omega)} (\|\gamma_2-g_S^{m}\|_{C^2(\Omega)}+ \|(h_1, h_2)\|_{\mathcal B_2(\Omega)})\Big)\\ 
 &\qquad + 2C_1^2 \|\gamma_2-\gamma_1\|^2_{C^0(\Omega)}. 
 \end{align*}
Noting for nonnegative $\beta$ and $\delta$, $w^2\leq \beta w + \delta^2$ implies $w\leq  \beta + \delta$, we conclude 
\begin{align} 
 \|f_2  -f_1\|_{H^2_\rho(\Omega, \gamma_1)} & \leq C \Big( \|\gamma_2-\gamma_1\|_{C^{2,\alpha}(\Omega)} \nonumber \\
 & \;\;\qquad +\|h_2-h_2\|_{\mathcal B_2(\Omega)}( \|\gamma_2-g_S^{m}\|_{C^2(\Omega)}+ \|(h_1, h_2)\|_{\mathcal B_2(\Omega)})\Big). \label{eq:f-diff-sob} 
 \end{align}

We have the following, using (\ref{eq:err1}), (\ref{eq:err2}), (\ref{eq:c-est-1}), (\ref{eq:f-diff-sob}) and the weighted Schauder estimate (\ref{eq:wtd-sch}) (where the constant ``$C$" can change from line to line)
\begin{align*}
\|h_2-h_1\|_{\mathcal B_2(\Omega)}&= \| \rho L_{\gamma_2}^*(f_2)- \rho L_{\gamma_1}^*(f_1)\|_{\mathcal B_2(\Omega)} \\
&\leq  \| \rho L_{\gamma_1}^*(f_2-f_1)\|_{\mathcal B_2(\Omega)}+ \| \rho (L_{\gamma_2}^*-  L_{\gamma_1}^*)(f_2)\|_{\mathcal B_2(\Omega)}\\
&\leq C\| f_2-f_1\|_{\mathcal B_4(\Omega)}+ \| \rho (L_{\gamma_2}^*-  L_{\gamma_1}^*)(f_2)\|_{\mathcal B_2(\Omega)}\\
&\leq C\Big( \| f_2-f_1\|_{\mathcal B_4(\Omega)}+ \|\gamma_2-\gamma_1\|_{C^{4,\alpha}(\Omega)} \|f_2\|_{\mathcal B_4(\Omega)}\Big)\\
&\leq C \Big( \| \rho^{-1} L_{\gamma_1}[\rho L_{\gamma_1}^*(f_2-f_1)]\|_{C^{0,\alpha}_{\phi, \phi^{4+\frac32} \rho^{\frac12}}}+\| f_2-f_1\|_{H_\rho^{2}(\Omega)}\\
&\qquad + \|\gamma_2-\gamma_1\|_{C^{4,\alpha}(\Omega)} \|f_2\|_{\mathcal B_4(\Omega)} \Big)\\
&\leq C\Big( \| \rho^{-1} (b(\gamma_2)-b(\gamma_1)) \zeta f^{m}\|_{\phi, \phi^{4+\frac32} \rho^{\frac12}}+ \|\rho^{-1} \mathscr E\|_{C^{0,\alpha}_{\phi, \phi^{4+\frac32} \rho^{\frac12}}}\\
&\qquad \| f_2-f_1\|_{H_\rho^{2}(\Omega)} + \|\gamma_2-\gamma_1\|_{C^{4,\alpha}(\Omega)} \|f_2\|_{\mathcal B_4(\Omega)} \Big) \\
&\leq C  \Big( \|\gamma_2-\gamma_1\|_{C^{4,\alpha}(\Omega)}\\ 
& \qquad + \|h_2-h_1\|_{\mathcal B_2(\Omega)} \Big( \|\gamma_2-g_S^{m}\|_{C^2(\Omega)}+ \|(h_1, h_2)\|_{\mathcal B_2(\Omega)}\Big) \Big).
\end{align*}
For $\gamma_1$ and $\gamma_2$ close enough to $g_S^{m}$, $R(\gamma_i)$ is close to zero, so that $\|h_i\|_{\mathcal B_2(\Omega)}\leq C \|R(\gamma_i)\|$ is sufficiently small so that we can absorb the second term on the right-hand side, to conclude 
$\|h_2-h_1\|_{\mathcal B_2(\Omega)}\leq C' \|\gamma_2-\gamma_1\|_{C^{4,\alpha}(\Omega)} $, as desired. 
\end{proof}

  \begin{remark} We observe that if $L^*_{\gamma_1}$ had non-trivial kernel, then as $R(\gamma_1)$ is constant (\cite[Thm. 1]{fm-sc-def}, cf.  \cite[Prop. 2.3]{cor:schw}  or \cite[Prop. 2.1]{cem}), it must vanish.     In that case, $h_1=0$, and then $R(\gamma)\rightarrow 0$ in $\mathcal B_0(\Omega)$ for $\gamma\rightarrow \gamma_1$ in $C^{2, \alpha}(\overline{\Omega})$ and with $R(\gamma)$ supported in a fixed subset $\overline{\Omega'}\subset \Omega$.  Via Lemma \ref{lem:eq:rjest}, we can then conclude $h(\gamma) \rightarrow 0=h_1$ in $\mathcal B_2(\Omega)$. 
 \end{remark}

\bibliographystyle{plain} 
\bibliography{bibliography.bib} 

\end{document}